%% file: geodesicPolygons_18.tex
\documentclass[a4paper,12pt,oneside]{amsart} 
\usepackage[left=3cm, right=3cm,top=3.5cm,bottom=3.5cm]{geometry}
\usepackage{algorithm, booktabs}
\usepackage{cite}
\include{macros}
\usepackage{mathtools}
\usepackage{lscape}
\usepackage{titlesec}
\titleformat{\section}
  {\centering \normalfont\fontsize{15}{15}\bfseries}{\thesection}{1em}{}
\titleformat{\subsection} 
  {\normalfont\fontsize{13}{13}\bfseries}{\thesubsection}{1em}{}
  \titleformat{\subsubsection} 
  {\normalfont\fontsize{13}{13}\itshape}{\thesubsubsection}{.5em}{}
\usepackage{enumitem}
\mathtoolsset{showonlyrefs}
\usepackage{amsfonts}
\usepackage{amsmath}
\usepackage{amssymb}
\usepackage{url}
\usepackage{faktor}
\usepackage{amscd}
\usepackage{array}
\usepackage{subfigure}
\usepackage{mathdots}
\usepackage{mathrsfs}
\usepackage{graphicx}
\usepackage{color}
\usepackage{pgf}
\usepackage{multirow}
\usepackage{tikz}
\usepackage{graphicx}

\newcommand{\N}{\ensuremath{\mathbb{N}}}
\newcommand{\mz}{Marcinkiewicz-Zygmund}
\newcommand{\sd}{\mathbb{S}^d}

\renewcommand{\S}{\ensuremath{\mathbb{S}}}

\newcommand{\R}{\ensuremath{\mathbb{R}}}

\def\3{\ss}

\newmuskip\pFqmuskip
\newcommand*\pFq[6][8]{
  \begingroup 
  \pFqmuskip=#1mu\relax
  \begingroup\lccode`\~=`\,
  \lowercase{\endgroup\let~}\pFqcomma
  {}_{#2}F_{#3}{\left(\genfrac..{0pt}{}{#4}{#5};#6\right)}%
  \endgroup
}
\newcommand*\pRegFq[6][8]{
  \begingroup 
  \pFqmuskip=#1mu\relax
  \begingroup\lccode`\~=`\,
  \lowercase{\endgroup\let~}\pFqcomma
  {}_{#2}\tilde{F}_{#3}{\left(\genfrac..{0pt}{}{#4}{#5};#6\right)}%
  \endgroup
}
\newcommand{\pFqcomma}{\mskip\pFqmuskip}

\DeclareMathOperator*{\diam}{diam}
\DeclareMathOperator*{\spann}{span}
\DeclareMathOperator{\dist}{dist}

\usepackage{etoolbox}
\makeatletter
\AtBeginEnvironment{align*}{\let\SetColor\@gobble \color{display}%
                            \Patchintertext}
\AtBeginEnvironment{align}{\let\SetColor\@gobble \color{display}\Patchintertext}
\AtBeginEnvironment{equation*}{\let\SetColor\@gobble \color{display}}
\AtBeginEnvironment{equation}{\let\SetColor\@gobble \color{display}}

\def\Patchintertext{\let\oldintertext@\intertext@
                    \def\intertext@{\oldintertext@\Patchintertext@}}
\def\Patchintertext@ {\let\oldintertext\intertext
                      \def\intertext ##1{\oldintertext
                      {\color{black}\let\SetColor\color ##1}}}
\makeatother
\everymath\expandafter{\the\everymath\SetColor{inline}}
\everydisplay\expandafter{\the\everydisplay\SetColor{display}}
\definecolor{display}{rgb}{0,0,.5}
\definecolor{inline}{rgb}{0,0,.3}
\let\SetColor\color
\usepackage{tcolorbox}
\tcbset{colback=gray!1,colframe=gray!20!white}

\begin{document}
\title[]{Geodesic cycles on the Sphere: $t$-designs and Marcinkiewicz-Zygmund Inequalities}
\author[M.~Ehler]{Martin Ehler}%\corref{cor1}}
\address[M.~Ehler]{University of Vienna,
Faculty of Mathematics, Vienna, Austria
}
\email{martin.ehler@univie.ac.at}
\author[K.~Gr\"ochenig]{Karlheinz Gr\"ochenig}%\corref{cor1}}
\address[K.~Gr\"ochenig]{University of Vienna,
Faculty of Mathematics, Vienna, Austria
}
\email{karlheinz.groechenig@univie.ac.at}
\author[C.~Karner]{Clemens Karner}%\corref{cor1}}
\address[C.~Karner]{University of Vienna,
Faculty of Mathematics, Vienna, Austria}
\email{clemens.karner@univie.ac.at}
%\thanks{K.\ G.\ was
%  supported in part by the  project P31887-N32  of the
%Austrian Science Fund (FWF)}

\begin{abstract}
A geodesic cycle is a closed curve that connects finitely many points along geodesics. We study geodesic cycles on the sphere in regard to their role in equal-weight quadrature rules and approximation. 
\end{abstract}

\maketitle
%---------------------------------------------------------------------------
%---------------------------------------------------------------------------
%---------------------------------------------------------------------------

\section{Introduction}
Geodesic cycles or chains on the sphere \(\mathbb{S}^d = \{x \in
\mathbb{R}^{d+1} : \|x\| = 1\}\) are configurations of finite geodesic
arcs and are studied in various fields of discrete and applied
mathematics. The examples in \cite{Viglietta23, Viglietta24} are
inspired by the art gallery problem in computational
geometry. Principal curves  formed by geodesic chains on the sphere
are applied in statistical data analysis \cite{Hauberg:2015dw,
  Lee:2021hj}. Geodesic chains are also used to approximate smooth
spherical curves effectively. Since geodesic chains can be described
by the coordinates of the endpoints of their arcs, the so-called
control points, they possess a handy and accessible description that
often facilitates an analytic approach to complex problems.

In this work, we study geodesic cycles that either form $t$-design curves \cite{EG:2023} or
satisfy \mz\ inequalities.  
The concept of $t$-design curves was recently proposed by us  in
\cite{EG:2023} as an extension of $t$-design points. To be specific,   
  a closed curve $\gamma: [0, 1] \to \mathbb{S}^d$ with arc length $\ell(\gamma)$ is called a spherical $t$-design \emph{curve} in \cite{EG:2023} if the path integral satisfies
\begin{equation} \label{tdesign1}
    \frac{1}{\ell(\gamma)} \int_{\gamma} f = \int_{\mathbb{S}^d} f\,, 
  \end{equation}
  for all polynomials in $d+1$ variables of degree not exceeding $t$. 
If equality in this definition is replaced by a norm equivalence of $L^p$-norms, one
obtains \mz\ inequalities which will be discussed below. 

Both $t$-designs and \mz\ inequalities are usually formulated for
point sets and possess a rich history. For comparison, recall that  
the defining property of  $t$-design
points is  the equal-weight quadrature rule
\begin{equation*}
    \frac{1}{n} \sum_{j=1}^{n} f(x_j) = \int_{\mathbb{S}^d} f\,,
\end{equation*}
for all algebraic polynomials $f$ in $d+1$ variables of total degree
at most $t$. 
Spherical design points have  been extensively
studied ever since the fundamental article  by Delsarte, Goethals, and Seidel in the 1970s
\cite{Delsarte:1977aa}. See
\cite{Womersley:2018we, Danev2001, Reznick:1995, Sloane:2003zp,
  HardinSloane96, Seidel:2001aa, Graf:2011lp, Goethals:81} for a
sample of contributions. A long list of numerical and some analytic
examples of $t$-design points on $\mathbb{S}^2$ are collected on the
websites \cite{Graef:website,Womersley:website} for a large range of
$t$. Another list for $\mathbb{S}^3$ is provided at
\cite{Womersley:website2}. We refer to \cite{HardinSloane96} for a
list of analytic $t$-designs on $\mathbb{S}^2$.

A major achievement in regard to $t$-design points is the proof of the
Korevaar-Meyers conjecture \cite{Korevaar93} by Bondarenko, Radchenko,
and Viazovska in 2013 \cite{Bondarenko:2011kx}, which had remained
open for 20 years. It verifies the existence of $t$-design points on
$\mathbb{S}^d$ with cardinality  $n\leq C_d t^d$ for some dimensional constant $C_d>0$.

Most questions about $t$-design points are meaningful, useful, and
interesting for $t$-design curves, but so far only very few results
are available. 
For instance, 
we know that there is a constant $c_d>0$ such that the arc length of
every $t$-design curve $\gamma$ in $\sd$ is  bounded below by 
	\begin{equation}\label{eq:asympt lower bound curve}
		\ell(\gamma) \geq c_d  t^{d-1}\,,
	\end{equation}
cf.~\cite[Thm.~2.2]{EG:2023}. The  analogue of the Korevaar-Meyers
conjecture for curves asks for the  existence of  a sequence of
$t$-design curves $(\gamma^{(t)})_{t\in\N}$ whose lengths grow at most
as $t^{d-1}$. The constructions of $t$-design curves in \cite{EG:2023} and
\cite{Lindblad1} affirm the Korevaar-Meyers conjecture for curves in
dimension $d=2$ and $d=3$. 
 So far only a handful of
explicit  examples  has been found.

A natural first idea for the construction of $t$-design curves is to
connect a set of $t$-design points along some curve and hope that the
resulting curve satisfies \eqref{tdesign1}. Proceeding in this manner,
one could build on the extensive collection of $t$-design points and
avoid building a theory from scratch.

Our objective is to explore several facets of this idea and obtain a
better grasp of what may be true and what not.

(i) We will discuss an example of a $t$-design curve with an explicit
analytic expression that contains $t$-design points in its
trace. However, it remains utterly  mysterious how to connect $t$-design
points  to obtain a $t$-design curve. The curve of Example~\ref{ex:odd} seems to be
a result of sheer luck. 

(ii) So the next  idea is to connect points along geodesic arcs, which
in $\S ^2$ are segments of great circles, and hope that the resulting
curve is a $t$-design curve.  In other words, we search $t$-design
curves in the form of geodesic cycles.   This idea fails already for the simplest $2$-design
points. Example~\ref{ex:S2} shows that the geodesic cycle obtained by
connecting the points of a regular tetrahedron  is not a
$2$-design.   

To improve on this idea, we use a geodesic cycle with $t$-design
points as control points for the initialization of a numerical
optimization with respect to the control points. Geometrically, we
deform the control points corresponding to  $t$-design points until an
error functional is zero, in which case we have a candidate for a
geodesic 
$t$-design cycle. In a second step we beautify the numerical
construction by reducing the number of parameters and proving
rigorously the existence of a $t$-design cycle. This procedure is
familiar in the construction of $t$-design points~\cite{Sloane:2003zp}, but seems
to be significantly more difficult for curve, both numerically and
analytically. So far, we have been successful to construct $t$-design
cycles only for $t=2$ and $t=3$. See Figures~\ref{fig:2chain} and
\ref{fig:3chain} for an illustration of these results.  

To the best of our knowledge, this construction  provides the only
 currently known geodesic $2$- and $3$-design cycles on $\mathbb{S}^2$
 that are free of self-intersections. 

 \begin{figure}
\subfigure[geodesic $2$-design cycle]{
\includegraphics[width=.3\textwidth]{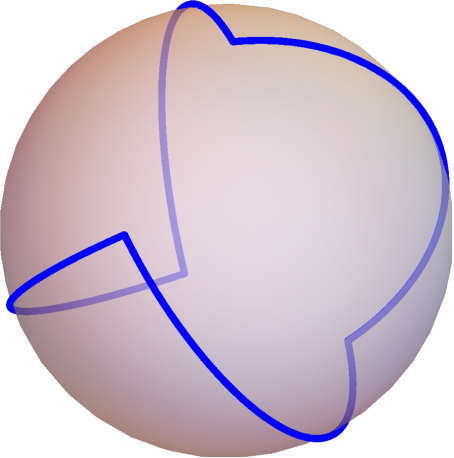}
\label{fig:2chain}}
\subfigure[geodesic $3$-design cycle]{                                                                                                 
\includegraphics[width=.3\textwidth]{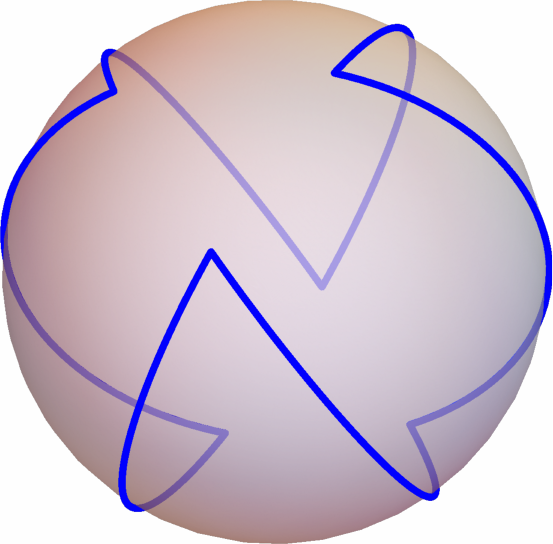}
\label{fig:3chain}}
\subfigure[geodesic $3$-design cycle]{
\includegraphics[width=.3\textwidth]{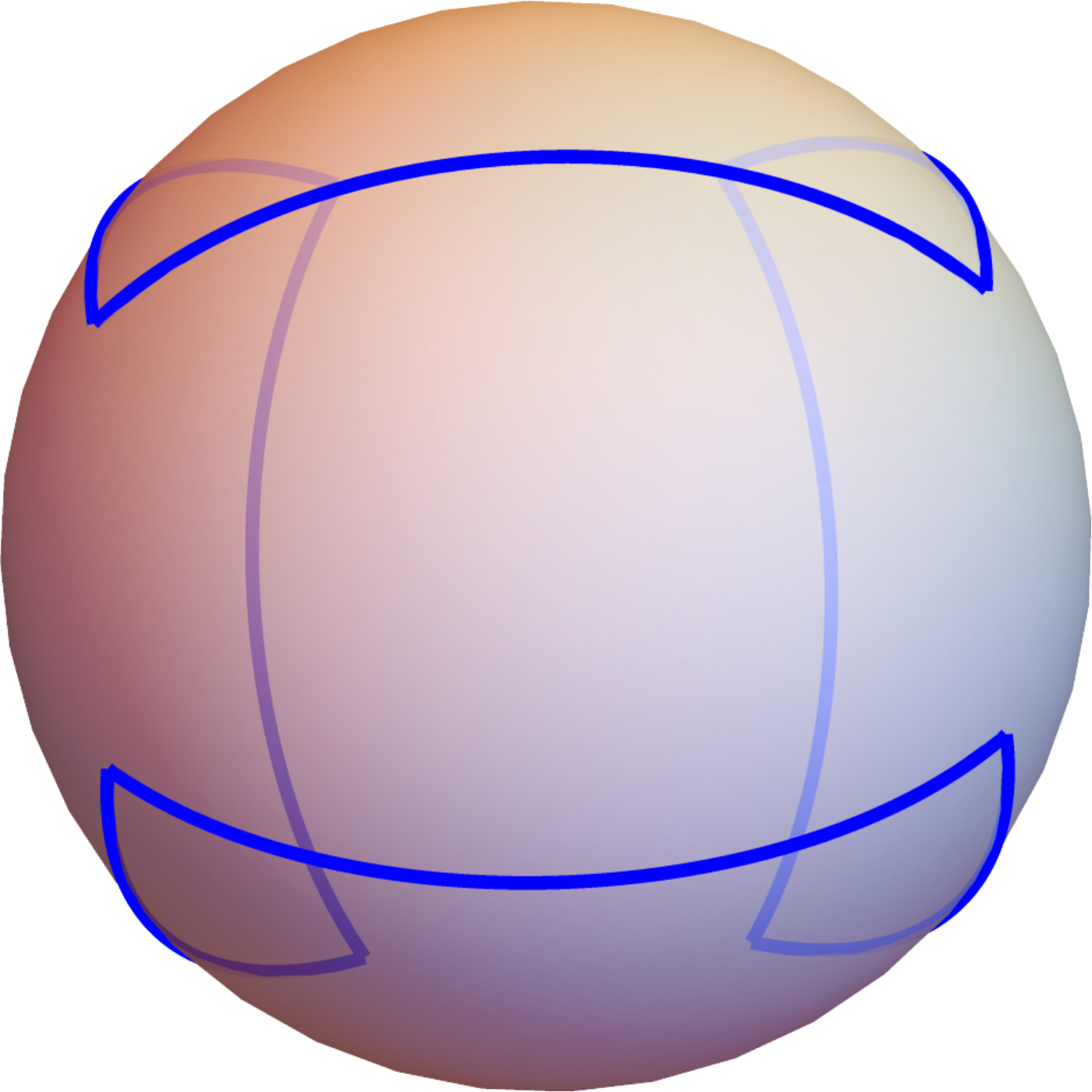}
\label{fig:3 design candidate cube}}
\caption{Visualizations of the geodesic $t$-design cycles constructed in Section \ref{sec:num}.}\label{fig:candidates}
\end{figure}

(iii) Although the initial idea fails, it raises the following question: What do we
actually obtain, if we connect $t$-design points by geodesic arcs? We
answer this question in the following sense (and in arbitrary
dimension). If the points are distributed sufficiently densely, as is
the case in the abstract construction of $t$-designs in~\cite{Bondarenko:2011kx}, then
the resulting geodesic cycle satisfies \mz\ inequalities. The analytic
result can be stated as follows.

\begin{tm}\label{thm:mz intro}
There are constants $0<A_d, B_d,C_d<\infty$ and a sequence of geodesic
cycles $(\gamma^{(t)})_{t\in\bN}$ in $\sd$ with the following
properties:

(i)  For all $p\in
[1,\infty ]$ and all degrees  $t\in\N$, the norm equivalence  
\begin{align*}
A_d\|f\|_{L^p(\sd)} \leq  \| f\|_{L^p(\gamma^{(t)})}   \leq B_d \|f\|_{L^p(\sd)}\,
\end{align*}
holds for all algebraic polynomials $f$ of $d+1$
variables of degree $t$, and

(ii) the length of the curves is bounded by
$$
\ell(\gamma^{(t)})  \leq C_{d} t^{d-1}\,.
$$
\end{tm}

Note that Theorem \ref{thm:mz intro} establishes \mz\ inequalities for sequences
of curves whose length matches  the lower  bounds
in \eqref{eq:asympt lower bound curve}.

\mz\ inequalities for points have been studied intensively in
approximation theory. For constructions and density theorems on the
sphere we refer to ~\cite{FilbirNew,Bondarenko:2011kx,Bondarenko:2015eu,marzo07,marzo08,marzo14,Mhaskar:2002ys}. In fact, our construction of \mz\
inequalities for curves is based on the existence of \mz\
points. Recently \mz\ inequalities have resurfaced in the context of
discretization of $L^p$-norms~\cite{temlyakov}. In this context,
Theorem~\ref{thm:mz intro} assures that the $L^p$-norm on a sphere can be captured
along a curve (instead of a finite set of points). To the  best of our
knowledge, Theorem~\ref{thm:mz intro} is the first result about \mz\ inequalities
for curves.

\subsection*{Outline}
The outline is as follows: In Section \ref{sec:design}, we recall the concepts of $t$-design points and curves and introduce geodesic $t$-design cycles. 
In Section \ref{sec:num}, we derive three geodesic $t$-design cycles
for $t=2,3$ that are free of self-intersections. The hidden part of
this work consists of numerical and symbolic computations. For
the symbolic computations we used Mathematica~\cite{mathematica}. Most formulas that come
with the epithet ``a computation leads to \dots `` were
obtained in this way. 
In Section \ref{sec:mz}, we prove Theorem~\ref{thm:mz intro} and
construct  a sequence of geodesic $t$-design cycles that satisfy \mz\ inequalities and whose lengths grow as $t^{d-1}$.

\section{Some Spherical $t$-Design Curves for Small $t$}\label{sec:design}
For $t\in\N$, let $\Pi_t$ be the collection of all polynomials with
real coefficients in $d+1$ variables of total degree at most $t$. The
unit $d$-sphere is denoted by 
\begin{equation*}
	\S^d=\{x\in\R ^{d+1}:\|x\|=1\}\,,\qquad d=2,3,\ldots.
\end{equation*}
We normalize its standard surface measure, so that $\int_{\sd}1 = 1$. 

As introduced by Delsarte, Goethals, and Seidel in the seventies \cite{Delsarte:1977aa,Seidel:2001aa}, a finite set $\{x_1,\ldots,x_n\}\subseteq\S^d$ is called a (spherical) \emph{$t$-design},  if
\begin{equation}\label{eq:t d def}
        \frac{1}{n}\sum_{j=1}^n f(x_j) = 	\int_{\S^d}
        f\, ,\qquad \text{for all } f\in\Pi_t\,.
\end{equation}
This concept gave rise to a rich field of research about explicit analytic, algebraic, and numerical constructions, \cite{Womersley:2018we,Danev2001,Harpe:2004,Reznick:1995,Sloane:2003zp} to name a few.

\subsection{Spherical $t$-design curves}
The analogous concept for curves has been introduced only recently in
\cite{EG:2023}. Here, the term curve always refers to a continuous,
piecewise smooth function $\gamma:[0,1]\rightarrow \sd$ of finite
arc-length  $\ell(\gamma)$ 
that is closed so that $\gamma(0)=\gamma(1)$.  We allow self
intersections, but in contrast to \cite{EG:2023},  the curve could
even traverse arcs  multiple times. Given a continuous function $f$ on
$\sd $ and a curve
$\gamma$,  the path integral is 
\begin{equation*}
\int_\gamma f = \int_0^1 f(\gamma(s))\|\dot{\gamma}(s)\|\mathrm{d}s\,
, 
\end{equation*}
and $\ell (\gamma ) = \int _\gamma 1$ is the arc length of $\gamma $. 
Given $f\in\Pi_t$ and a curve $\gamma$, we compare $\int _\gamma f$
with the integral $\int _{\sd } f$  over the entire sphere.

We call $\gamma$ a (spherical) $t$-design curve if 
\begin{equation*}
\frac{1}{\ell(\gamma)}\int_\gamma f = 	\int_{\S^d}
        f \,,\qquad \text{for all } f\in\Pi_t\,.
\end{equation*}

Some initial examples of smooth  $1$, $2$, and $3$-design curves on $\S^2$ have been derived in \cite{EG:2023}. 
 \begin{exmp}[Smooth $t$-design curves in $\mathbb{S}^2$]\label{ex:S2}
 For $t=1,2,3$, consider the curves $\gamma^{(t,a)}:[0,1]\rightarrow\S^2$, 
\begin{equation}\label{eq:gamma curve 0}
	\gamma^{(t,a)}(s):=
		\begin{pmatrix}
			a\cos(2\pi s)+(1-a)\cos(2\pi (2t-1) s)\\
			a\sin(2\pi s)-(1-a)\sin(2\pi (2t-1) s)\\
			2\sqrt{a(1-a)}\sin(2\pi t s)
		\end{pmatrix}.
\end{equation}
The curve $\gamma^{(1,a)}$ is a great circle. Every great circle on
the sphere $\S^d$ is a $1$-design curve and it is easy to show that
great circles are  shortest among all $1$-design curves. 

The $t$-design property of $\gamma ^{(t,a)}$  was proved in \cite[Prop.~3.1]{EG:2023}:
\emph{for  $t=2$ and $t=3$ there exist parameters
  $a_2,a_3\in(\frac{1}{2},1)$ such that $\gamma^{(2,a_2)}$ and
  $\gamma^{(3,a_3)}$ are spherical $2$- and $3$-design curves,
  respectively. }

Both curves are simple, meaning they have no self-intersections. As a result, they divide the sphere into two distinct regions, as shown in Figure \ref{fig:approx 3 design}. Interestingly, we observe that these two regions have equal areas.
\begin{prop}
For $t=2$, $t=3$, and $a\in [0,1]$, $\gamma ^{(t,a)}$ partitions $\S
^2$ into two regions of equal area.   
\end{prop}

\begin{proof}
  Due to our normalization $|\mathbb{S}^2|=1$, this  means that the area of both regions is $\frac{1}{2}$. 

To verify this  observation rigorously, we recall that the integral of the geodesic curvature $k_g$ along a smooth curve $\gamma:[0,1]\rightarrow\S^2$ may be computed by 
\begin{equation*}
\int_\gamma k_g =\frac{1}{4\pi} \int_0^1 \frac{\langle \ddot{\gamma}(s),\gamma(s)\times\dot{\gamma}(s)\rangle}{\|\dot{\gamma}(s)\|^2}\mathrm{d}s\,,
\end{equation*}
cf.~\cite[Sections 17.4 and 27.1]{Gray}. 
The Gauss-Bonnet Theorem applied to curves on $\S^2$ yields that the enclosed area $A(\gamma)$ is determined by 
\begin{equation}\label{eq:GB}
A(\gamma) = \frac{1}{2} - \int_\gamma k_g\,.
\end{equation}
Using the parametrization $\gamma=\gamma^{(2,a)}$, elementary Mathematica computations reveal 
{\small 
\begin{equation*}
k_g = \frac{\langle \ddot{\gamma}(s),\gamma(s)\times \dot{\gamma}(s)\rangle}{\|\dot{\gamma}(s)\|^2} = 
-4 \pi  \sqrt{(1-a) a} \sin (4 \pi  s)\frac{2 a^2-2 (a-1) a \cos (8 \pi  s)-14 a+15}{2 a^2-2 (a-1) a \cos (8 \pi  s)-10 a+9}\,.
\end{equation*}
}
Since $\sin(4\pi s)$ is multiplied by a function whose period is $1/4$, the standard formula $\sin(4\pi (s+\frac{1}{4}))=-\sin(4\pi s)$ implies that the integral over $[0,1]$ vanishes.
Hence, we derive $A(\gamma)=\frac{1}{2}$ and the normalization $|\mathbb{S}^2|=1$ implies that $\gamma$ divides $\mathbb{S}^2$ into two regions of equal area. 

The parametrization $\gamma=\gamma^{(3,a)}$ is dealt with in a similar fashion, and we obtain
{\small 
\begin{equation*}
k_g = \frac{\langle \ddot{\gamma}(s),\gamma(s)\times \dot{\gamma}(s)\rangle}{\|\dot{\gamma}(s)\|^2} = 
-8 \pi  \sqrt{(1-a) a} \sin (6 \pi  s)\frac{2 a^2-2 (a-1) a \cos (12 \pi  s)-11 a+10}{2 a^2-2 (a-1) a \cos (12 \pi  s)-8 a+25/4}\,,
\end{equation*}
}
and again we  conclude that $\int_\gamma k_g=0$. 
\end{proof}

Although we are dealing with a family of smooth curves, there is a
hidden relation to $t$-design points. To see this, recall that the
vertices of a smooth curve on the sphere are the local extrema of its
geodesic curvature. For $t=2$ and $\gamma ^{(2,a)}$, there are four vertices located at
$s_j=\frac{2j-1}{8}$, $j=1,\ldots,4$. For the parameter $a_S=\frac{1}{2}+\frac{1}{\sqrt{6}}$, we have checked that the $4$
points $\{\gamma^{(2,a_S)}(\tfrac{2j-1}{8})\}_{j=1}^4$ form a
$2$-design given by the vertices of a (regular) tetrahedron, but 
$\gamma^{(2,a_S)}$ is not a $2$-design curve.

By contrast, for the special parameter $a_2$ the corresponding curve
is a $2$-design curve, but its vertices are not $2$-design points and we did not find any other $2$-design points that lie on the curve.

Likewise, for $t=3$, there are six vertices located at
$s_j=\frac{2j-1}{12}$, $j=1,\ldots,6$. The parameter choice
$a_O=\frac{1}{2}$ leads to the six points
$\{\gamma^{(3,a_O)}(\tfrac{2j-1}{12})\}_{j=1}^6$, and we checked that
they form a $3$-design as the vertices of an octahedron, but the
corresponding curve is not a $3$-design curve. Again, for the special
parameter $a_3 \neq a_O$ one obtains a $3$-design curve, but its
vertex set does not contain $3$-design points and we did not find any other $3$-design points that lie on the curve.

\begin{figure}
\subfigure[$\gamma^{(2,a_2)}$ with $a_2 \approx 0.7778$]{
\includegraphics[width=.3\textwidth]{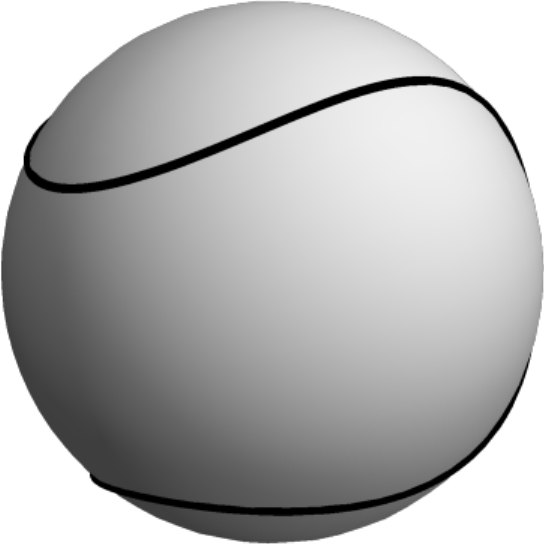}\label{fig:smooth 2}}\hspace{6ex}
\subfigure[$\gamma^{(3,a_3)}$ with $a_3\approx 0.7660$]{
\includegraphics[width=.3\textwidth]{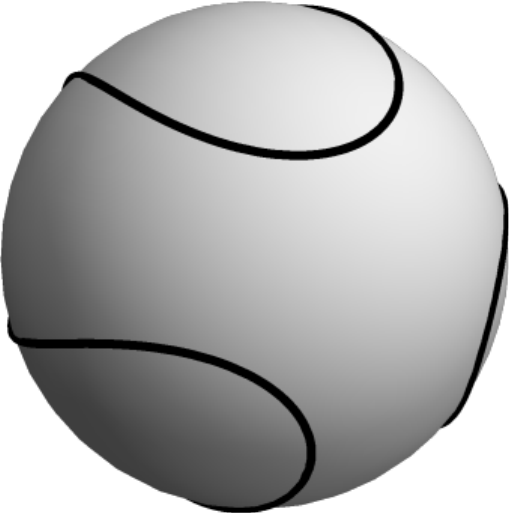}\label{fig:smooth 3}}
\caption{Curves in Example \ref{ex:S2}.}
\label{fig:approx 3 design}
\end{figure}
\end{exmp}
At this time, the curves $\gamma^{(2,a_2)}$ and $\gamma^{(3,a_3)}$ are the
only known  smooth, simple $2$- and $3$-design curves. So far, we   do not know any smooth $4$-design curve in $\S^2$. 

Some examples of  $2$- and $3$-design curves in
higher-dimensional spheres can be constructed as follows. % $\S^{2m-1}$.

\begin{exmp}[Smooth $t$-design curves in $\mathbb{S}^{2m-1}$]\label{ex:odd}
Let $c(s)=\begin{pmatrix} \cos(2\pi s)\\ \sin(2\pi s) \end{pmatrix}$
be  the circle traversed counter-clockwise, and consider the curves $\gamma^{(1)}(s)=\frac{1}{\sqrt{m}} \left(c(s),\ldots,c(s)\right)\in\mathbb{S}^{2m-1}$ and 
\begin{equation}\label{eq:1 curve}
\gamma^{(2)}(s):=\frac{1}{\sqrt{m}} 
		\begin{pmatrix}
			c(s)\\
			c(2s)\\
			\vdots\\
			c(ms)
		\end{pmatrix},\qquad 
	\gamma^{(3)}(s):=\frac{1}{\sqrt{m}} 
		\begin{pmatrix}
			c(s)\\
			c(3 s)\\
			\vdots\\
			c((2m-1)s)
		\end{pmatrix}\,.
\end{equation}
Then $\gamma^{(t)}$ is a $t$-design curve for $t=1,2,3$, and $\gamma^{(1)}$ is a great circle. 

We have checked that the points $\left\{\gamma^{(2)}(\tfrac{j}{2m+1})\right\}_{j=1}^{2m+1}$ are a $2$-design, namely the vertices of the $2m$-tetrahedron. The points $\left\{\gamma^{(3)}(\tfrac{j}{4m})\right\}_{j=1}^{4m}$ are a $3$-design and correspond to the vertices of the $2m$-dimensional cross-polytope (hyperoctahedron). 
\end{exmp}

Let us look at some subtle differences in Examples~\ref{ex:S2}
and~\ref{ex:odd}. The $t$-design curves in Examples \ref{ex:odd} connect $t$-design
points by a suitable curve. In this sense they confirm the idea that
one $t$-design curve could be obtained by suitably connecting
$t$-design points. By contrast,   in Example \ref{ex:S2} with special
values $a_2, a_3$, the
$t$-design curves pass through points that  are close to $t$-design
points, but these curves do \emph{not} contain $t$-design points.
In both cases, by sheer luck, we had an   analytic expression
for the curve. It remains unclear how to find such an expression.

We therefore now turn to geodesic cycles, which are geometrically
simpler. In geometry, the motion  along a geodesic is the canonical
way to connect two points.

\subsection{$t$-design curves consisting of geodesic cycles}

 In this section we start with a set of points, preferably $t$-design
 points, and then connect them with the goal of obtaining a $t$-design
 curve. The idea is to connect them  along the simplest possible curve
 in $\S ^2$, namely geodesic arcs.  This idea leads to the notion of geodesic chains and cycles.  

A geodesic chain   is a   curve on the sphere $\mathbb{S}^d$ that
connects a finite set of control points by geodesic arcs, and a
geodesic cycle is a \emph{closed} geodesic chain. Geodesic
chains  are the analogue of polygonal curves in Euclidean space.    

The length of the geodesic arc connecting $x$ and $y$ on $\S^d$ is measured by the rotation invariant metric on $\S ^d$, 
\begin{equation}\label{eq:dist def}
\dist (x,y) = \arccos\langle x,y \rangle\,.
\end{equation}
Unless $y\neq \pm x$,  such an arc is parametrized by $\gamma_{x,y}:[0,1]\rightarrow\S^d$,
\begin{equation*}
\gamma_{x,y}(s)= \frac{\sin((1-s)\dist (x,y))}{\sin(\dist (x,y))} x + \frac{\sin(s\dist (x,y))}{\sin(\dist (x,y))}y\,.
\end{equation*}
We note that $\sin(\dist (x,y))=\sqrt{1-\langle x,y \rangle^2}$ and that this parametrization has constant speed $\|\dot{\gamma}_{x,y}(s)\| = \dist (x,y)$.

The control points $x_1,\ldots,x_n\in \S^d$ induce the geodesic cycle
$\gamma=(\gamma_{x_j,x_{j+1}})_{j=1}^n$, where we additionally put
$x_{n+1}:=x_1$ to obtain a closed curve. We do not require that the control points are pairwise distinct and hence we allow geodesic arcs to occur multiple times. 

The path  integral over $\gamma=(\gamma_{x_j,x_{j+1}})_{j=1}^n$ is 
\begin{align*}
\int_\gamma f & = \sum_{j=1}^n \int_{\gamma_{x_j, x_{j+1}}}f\,.
\end{align*}
The length of $\gamma$ is $\ell(\gamma)=\sum_{j=1}^n
\ell(\gamma_{x_j,x_{j+1}}) = \sum_{j=1}^n \dist(x_j,x_{j+1})$ and is a
function of the control points only.

We call $\gamma$ a geodesic $t$-design cycle (or simply $t$-design curve) if 
\begin{equation*}
\frac{1}{\ell(\gamma)} \int f = \int_{\S^d} f\, ,\qquad \text{for all } f\in \Pi_t\,.
\end{equation*}

A natural idea for the construction of geodesic $t$-design cycles is
to use control points that are themselves $t$-design points. As
mentioned earlier, the vertices of the  regular tetrahedron
$\{x_1,x_2,x_3,x_4\}$ are $2$-design points on
$\mathbb{S}^2$. However, the geodesic cycle induced by these four
control points is not a $2$-design cycle and not even a $1$-design
cycle.

To obtain $t$-design curves, we perturb the control points and thus
deform a given cycle. For this,  we set up an iterative numerical
optimization that we initialize with $t$-design control points.

\section{Numerical constructions and beautification}\label{sec:num}

We perform a numerical optimization to derive candidates of simple geodesic $t$-design cycles. Called beautification 
in \cite{Sloane:2003zp}, we eventually aim to derive an analytic description of these candidates and prove that they satisfy the respective design properties. 

It seems  that the process of beautification  is more demanding for
curves  than for points. While the numerical minimization is comparable, finding the analytic description appears much harder. We are able to complete this beautification procedure for $t=2$ and $t=3$, but we obtain two numerical candidates of $5$-design cycles that we have not yet been able to beautify.

\subsection{Two-step procedure: numerical minimization and beautification}\label{sec:two}
A closed curve $\gamma$ is a $t$-design if and only if the linear form 
\begin{equation*}
Lf=\int_{\mathbb{S}^2} f - \frac{1}{\ell(\gamma)}\int_\gamma f
\end{equation*}
vanishes on $\Pi_t$. The norm of $L$
\begin{equation}\label{eq:wce}
\|L\|_t:=\sup_{\substack{f\in\Pi_t \\ \|f\|_{L^2(\mathbb{S}^2)}\leq 1}} \left|\int_{\mathbb{S}^2}f - \frac{1}{\ell(\gamma)}\int_\gamma f \right| 
\end{equation}
is the worst case integration error on $\Pi_t$. For numerical
optimization, we may use the following expression for the norm of $L$,
see \cite{Ehler:2019aa,SloanWom,Graf:2011lp} for related formulas
for points.
\begin{lemma}\label{lemma:num opt}
Let $\{P_l : l=0,1,\ldots\}$ be the family of Legendre polynomials,
normalized by $P_l(1)=1$. Given a curve $\gamma$, the norm of $L$ is
given by 
\begin{equation}\label{eq:L norm to be minimized}
\|L\|^2_t = \sum_{l=1}^ t \frac{2l+1}{|\ell(\gamma)|^2}\int_0^1 \int_0^1 P_l(\langle\gamma(r),\gamma(s)\rangle)\|\dot{\gamma}(r)\| \|\dot{\gamma}(s)\|\mathrm{d}r\mathrm{d}s\,.
\end{equation}
As a consequence, $\gamma$ is a $t$-design curve if and only if 
\begin{equation*}
\int_0^1 \int_0^1 P_l(\langle\gamma(r),\gamma(s)\rangle)\|\dot{\gamma}(r)\| \|\dot{\gamma}(s)\|\mathrm{d}r\mathrm{d}s = 0\,,\qquad l=1,\ldots,t\,.
\end{equation*}
\end{lemma}
\begin{proof}
We only need to consider the restriction of the polynomials $\Pi_t$ to
the sphere. For an arbitrary orthonormal basis $\{\varphi_k\}$ of
$\Pi_t|_{\mathbb{S}^2}\subseteq L^2(\S^2)$, the Riesz representative of $L$ is $v_L=\sum_k (L\varphi_k)\varphi_k\in \Pi_t|_{\mathbb{S}^2}$ and the norm of $L$ satisfies 
\begin{equation}\label{eq:norm of L}
\|L\|_t = \|v_L\|=\left(\sum_k| L\varphi_k|^2\right)^{1/2} \, .
\end{equation}
We choose the natural orthonormal basis for $\Pi_t|_{\mathbb{S}^2}$, namely the spherical harmonics up to degree $t$ denoted by  
$\{Y_{l,m}:|m|\leq l\,, \, l=0,\ldots,t\}$. According to \eqref{eq:norm of L} the integration error is 
\begin{equation*}
\|L\|^2_t = \sum_{l=0}^ t \sum_{m=-l}^{l} \left| \int_{\mathbb{S}^2}Y_{l,m}-\frac{1}{\ell(\gamma)}\int_\gamma Y_{l,m}  \right|^2\,.
\end{equation*}
Since $Y_{0,0}\equiv 1$ and thus $\int_{\S^2} Y_{l,m} = \langle Y_{l,m},Y_{0,0}\rangle_{L^2(\S^2)}=0$ for $l>0$ and $\frac{1}{\ell(\gamma)}\int_\gamma 1 = \int_{\S^2} 1$ by definition, the integration error can be expressed as
\begin{align*}
\|L\|_t^2 &= 
 \sum_{l=1}^ t \sum_{m=-l}^{l} \left|\frac{1}{\ell(\gamma)}\int_\gamma Y_{l,m}\right|^2\\
& = \sum_{l=1}^ t \sum_{m=-l}^{l} \frac{1}{|\ell(\gamma)|^2}\int_0^1 \int_0^1 Y_{l,m}(\gamma(r))Y_{l,m}(\gamma(s))  \|\dot{\gamma}(r)\| \|\dot{\gamma}(s)\|\mathrm{d}r\mathrm{d}s\,.
\end{align*}
We use the well-known addition formula of the spherical harmonics in terms of Legendre polynomials,
\begin{equation}\label{eq:add formula}
 \sum_{m=-l}^{l}Y_{l,m}(x)Y_{l,m}(y) = (2l+1) P_l(\langle x,y\rangle)
 \, ,\qquad \text{for all }x,y\in\mathbb{S}^2\,,
\end{equation}
and  further rewrite the integration error as
\begin{equation*}
\|L\|_t^2= 
\sum_{l=1}^ t \frac{2l+1}{|\ell(\gamma)|^2}\int_0^1 \int_0^1 P_l(\langle\gamma(r),\gamma(s)\rangle)\|\dot{\gamma}(r)\| \|\dot{\gamma}(s)\|\mathrm{d}r\mathrm{d}s\,.
\end{equation*}

The addition formula yields
\begin{equation*}
(2l+1) \int_0^1 \int_0^1
P_l(\langle\gamma(r),\gamma(s)\rangle)\|\dot{\gamma}(r)\|
\|\dot{\gamma}(s)\|\mathrm{d}r\mathrm{d}s =
\left(\int_0^1\sum_{m=-l}^{l}Y_{l,m}(\gamma(r))\|\dot{\gamma}(r)\|
  \mathrm{d}r \right)^2\geq 0\, .
\end{equation*}
Consequently $\|L\|_t = 0$, if and only if the left-hand side vanishes
for $l=1, \dots , t$. 
\end{proof}
Our  recipe for the construction  of geodesic $t$-design cycles
consists of two steps, first a  numerical minimization and then a  beautification. 
As the first step, the \emph{numerical} process works as follows: 
\begin{itemize}
\item[(i)] The goal of the numerical procedure is to minimize the error
  functional ~\eqref{eq:L norm to be minimized}.
  A natural initialization is a geodesic cycle on $\S^2$
  that connects $t$-design points. Indeed,  
  the vertices of the Platonic solids form $t$-design points for $t=2,3,5$, respectively, and we initialize the minimization algorithm with their spherical Hamiltonian cycles as depicted in Figure \ref{fig:Hamil}.
\begin{figure}
\hspace{-0.6cm}
\subfigure[tetrahedron]{
\includegraphics[width=.19\textwidth]{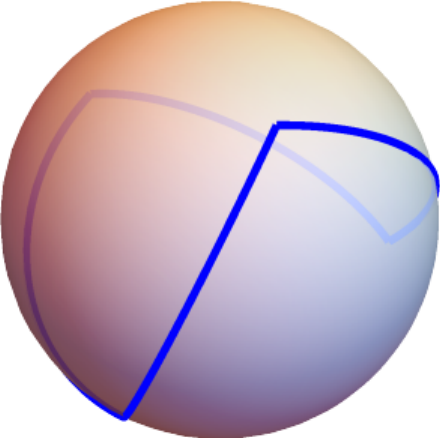}\label{fig:simplex init}}
\subfigure[octahedron]{
\includegraphics[width=.19\textwidth]{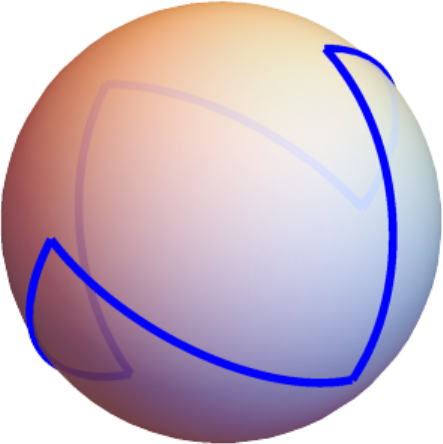}\label{fig:octahedron init}}
\subfigure[cube]{
\includegraphics[width=.19\textwidth]{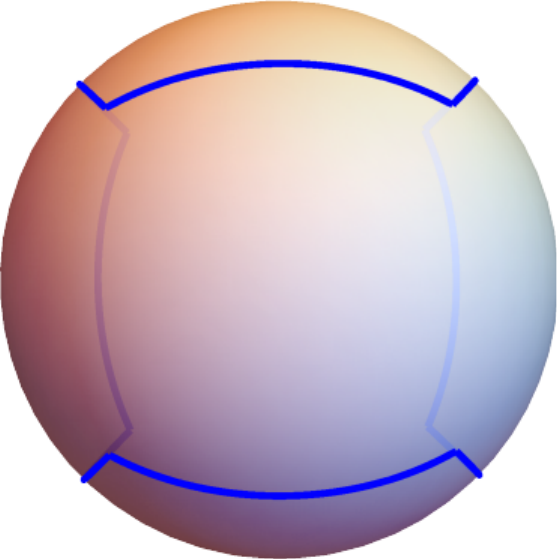}\label{fig:cube init}}
\subfigure[icosahedron]{
\includegraphics[width=.19\textwidth]{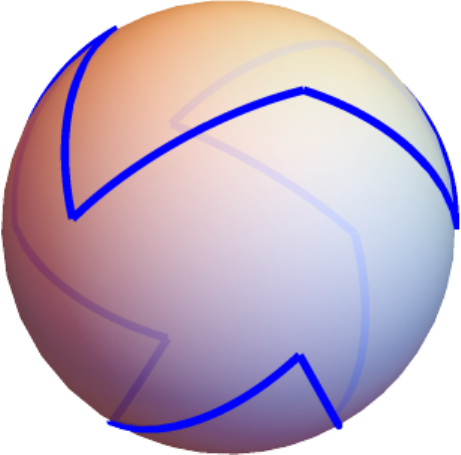}}
\subfigure[dodecahedron]{
\includegraphics[width=.19\textwidth]{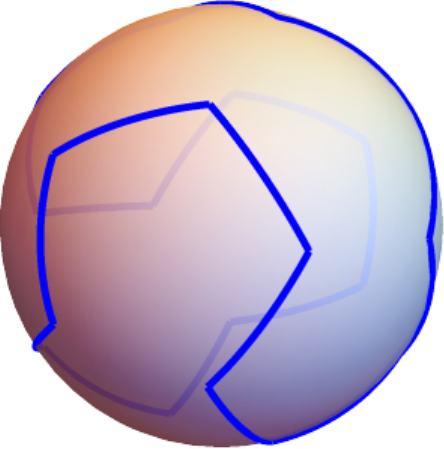}}
\caption{The platonic solids are the five regular polytopes in
  $\R^3$. The vertices of the tetrahedron form a $2$-design, the vertices
  of the octahedron and the cube yield a $3$-design, and the
  icosahedron and the dodecahedron form a $5$-design. They are
  regular convex polytopes in $\R^3$ and admit Hamiltonian cycles,
  i.e., cycles of edges that pass each vertex exactly once. Their
  projections onto the sphere lead to the geodesic cycles depicted
  here. We emphasize that the Hamiltonian cycles (a) -- (e)  do not  form geodesic $1$-design cycles. 
}\label{fig:Hamil}
\end{figure}

For geodesic cycles, the error functional \eqref{eq:L norm to be
  minimized} is a function of the control points, and is thus a
function of finitely many parameters (twice the number of control
points).   
\item[(ii)] We use a suitable  iterative optimization algorithm,  such
  as gradient descent, steepest descent, or some conjugate gradient
  methods,  to minimize the error functional. Geometrically, each iteration of the descent algorithm yields a new set of parameters, or in other words, a new set of control points with a corresponding geodesic cycle. 
\item[(iii)] The iteration stops when we have found a set of control
  points for which the error functional \eqref{eq:L norm to be
    minimized} vanishes up to machine precision. By Lemma
  \ref{lemma:num opt} this numerical solution yields a promising
  candidate of a geodesic $t$-design cycle.  
\end{itemize}

As the  second step, the \emph{beautification} process refines the numerical candidate into an exact geodesic $t$-design cycle. 
\begin{itemize}
\item[(iv)] We identify a reduced set of parameters -- much fewer than
  twice the number of control points -- that still ensures a
  continuous transition from the  initialization to the solution. This selection is based on a visual comparison of the control points between the initial configuration and the numerical solution.
\item[(v)] The $t$-design property is formulated in the reduced parametrization as a system of nonlinear equations.
\item[(vi)] We prove that this system of equations is solvable and
  yields a $t$-design curve. 
\end{itemize}
 In the subsequent sections we will report on our progress in regard to the numerical process and the
beautification process.

\subsection{Numerical candidates of geodesic $t$-design cycles}
(i) Figure~\ref{fig:candidates} in the introduction shows the results of the numerical
optimization, when the procedure was initialized with a geodesic cycle
based on one of the Platonic solids (a) -- (c) in
Figure~\ref{fig:Hamil}. 

(ii) If we start with the initial  cycle based on the icosahedron (d) for
$t=5$, the error functional~\eqref{eq:L norm to be
    minimized} does not vanish.  
This initialization cycle consists of  twelve  arcs,
and our numerical computations suggest that there does not exist any geodesic
$5$-design cycle with only twelve arcs. 

(iii) We proceed with the twenty vertices of the dodecahedron (e). The
initial  cycle has twenty parts and we obtain a numerical candidate of a geodesic $5$-design cycle shown in  Figure \ref{fig:3 design candidate dodecahedron}. 
\begin{figure}[h]
\subfigure[$20$ arcs]{
\includegraphics[width=.33\textwidth]{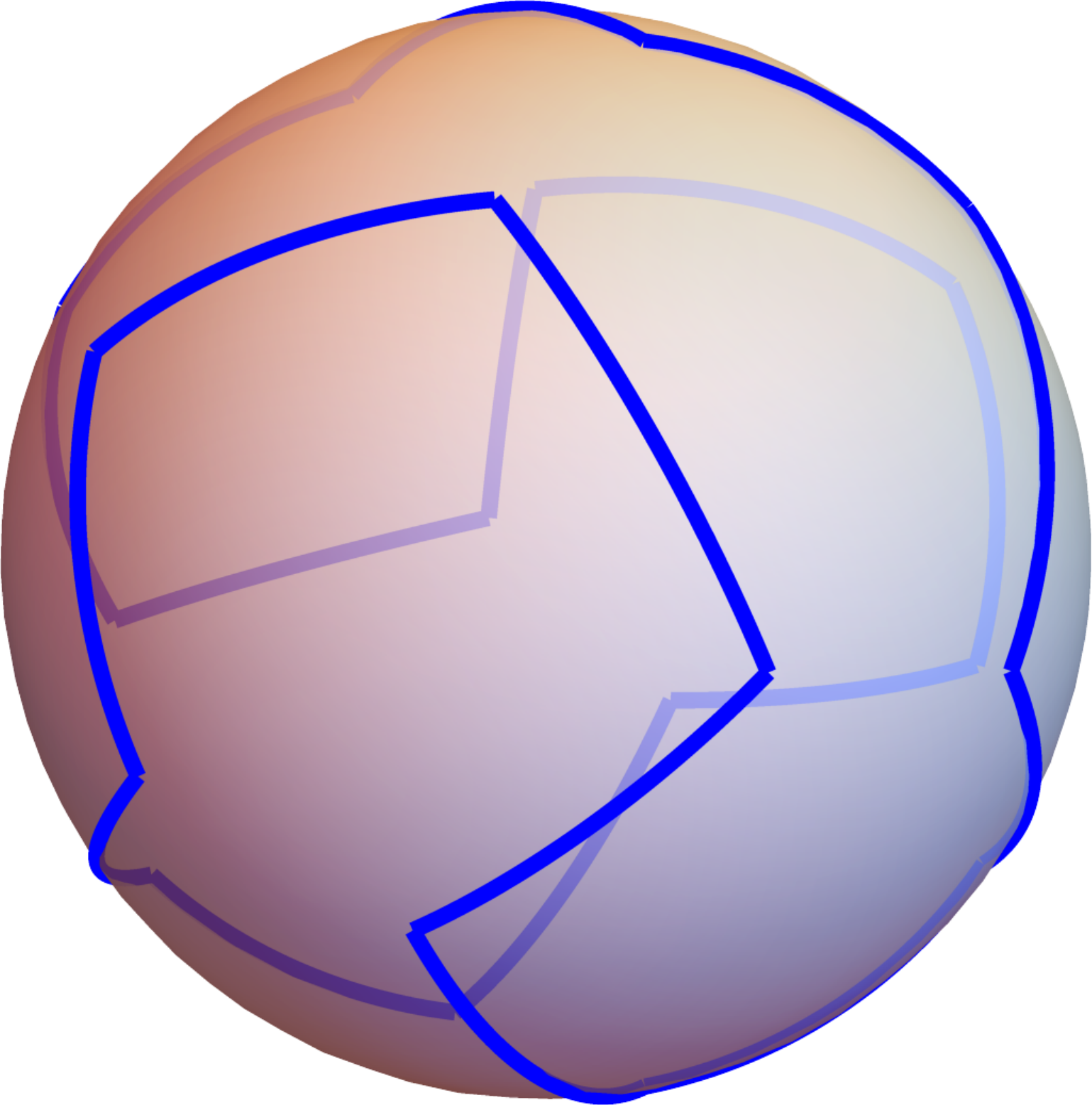}
\label{fig:3 design candidate dodecahedron}}\hspace{1cm}
\subfigure[$18$ arcs]{
\includegraphics[width=.33\textwidth]{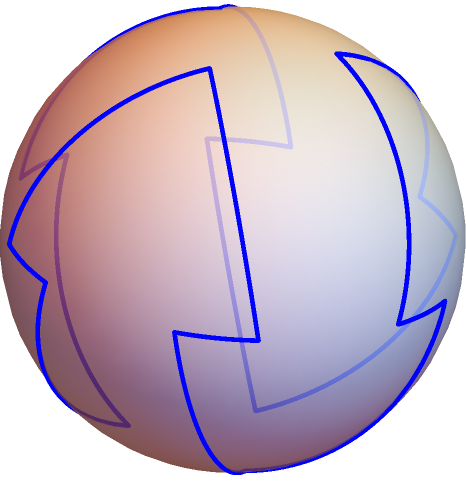}
\label{fig:3 design candidate 18 point design}}
\caption{Numerical candidates of geodesic $5$-design cycles.}\label{fig:cand}
\end{figure}

(iv) To explore if there is also a numerical candidate of a geodesic
$5$-design cycle with fewer than twenty arcs, we need to initialize
the numerical scheme with  fewer control points. 
According to \cite{HardinSloane96,Reznick:1995}, there exist $5$-design points in $\S^2$ if the number of points equals $12$, $16$, $18$, $20$, or any integer $\geq 22$. We have not been successfull with sixteen arcs, but when initializing by suitably connecting $18$ points of a spherical $5$-design, we derive the candidate of a geodesic $5$-design cycle with $18$ parts in Figure \ref{fig:3 design candidate 18 point design}.

\subsection{Beautification of some candidates}\label{sec:beautify}
We now start the beautification process. For $t$-design points, the beautifying
procedure usually starts with the computation of all pairwise inner
products. The goal is  to eventually identify (or ''guess'') a finite group of rotations and reflections, so that the points are (unions of) orbits of this symmetry group \cite{Sloane:2003zp}. 

In the  setting of curves, the inner products between the control
points do not reveal sufficient information to make an educated guess
about any group orbits. Instead, as outlined in the beautification
process in Section \ref{sec:two}, we efficiently parameterize the
control points, rewrite the $t$-design conditions in terms of these
parameters, and eventually argue that the parameter equations are
solvable.

\subsubsection{Beautifying the candidates resulting from the tetrahedron and the octahedron}

We show how the beautification is done for the candidates in Figures \ref{fig:2chain}
and \ref{fig:3chain}.

\subsubsection*{The tetrahedron}

Connecting the four vertices of the tetrahedron yield a geodesic cycle for
the initialization of the numerical optimization process, see Figure
\ref{fig:simplex init}. The four control points require eight
parameters in general. By comparing it with the numerical solution of
the minimization, we identified  a one-parameter family of control
points that facilitate a  continuous transition from the vertices of
the tetrahedron to the control points of the numerical
minimizer. Specifically,  for $a\in(0,\frac{\pi}{2}]$, we introduce
the control points 
\begin{align}\label{eq:points t=2 simple}
{\small
x_1=\begin{pmatrix}
\sin(a)\\0\\\cos(a)
\end{pmatrix},
\quad 
x_2=\begin{pmatrix}
0\\ \sin(a)\\ -\cos(a)
\end{pmatrix},
\quad 
x_3=\begin{pmatrix}
-\sin(a)\\0\\ \cos(a)
\end{pmatrix},
\quad 
x_4=\begin{pmatrix}
0\\ -\sin(a) \\ -\cos(a)
\end{pmatrix}}\,.
\end{align}
The induced geodesic cycle is denoted by $\Gamma^{(2,a)}$. The points
$x_1$ and $x_3$ are at distance $d(x_1,e_3) = \arccos \langle x_1, e_3
\rangle = a = d(x_3,e_3)$ from the north pole $e_3= (0,0,1)^\top$, and
a rotation  of their antipodals by $\pi /2$ degrees yields $x_2$ and $x_4$, see Figures \ref{fig:2a0}, \ref{fig:2circb}, and \ref{fig:2circ}.

The choice $a=\arctan(\sqrt{2})\approx 0.9553$ yields the vertices of
the regular  tetrahedron, whereas  $a\approx 1/2$ leads to a geodesic cycle that resembles the candidate in Figure \ref{fig:2chain}. 

The extreme cases are  $a=\frac{\pi}{2}$, for which  we obtain a great
circle,  and $a\rightarrow 0$, which  approximates two great circles
perpendicular to each other, cf.~Figure \ref{fig:2a0} and
\ref{fig:2circ}.  

\subsubsection*{The octahedron}

Connecting the six vertices of the octahedron yield a geodesic cycle
with $6$ control points  shown in Figure \ref{fig:octahedron
  init}. Again, we reduce the twelve parameters corresponding to the
six control points to a one-parameter family of six points.
Specifically, we parametrize the control points  by 
\begin{equation}\label{eq:y}
\begin{aligned}
y_1&=\begin{pmatrix}
\sin(a)\\0\\\cos(a)
\end{pmatrix},
&
y_2 &= \begin{pmatrix}
\frac{1}{2}\sin(a)\\
\frac{\sqrt{3}}{2}\sin(a)\\
-\cos(a)
\end{pmatrix},
&
y_3&=\begin{pmatrix}
-\frac{1}{2}\sin(a)\\
\frac{\sqrt{3}}{2}\sin(a)\\
\cos(a)
\end{pmatrix},\\
y_4 &= -y_1,
& y_5 &= -y_2,
& y_6 &= -y_3\, , 
\end{aligned}
\end{equation}
for $a\in (0,\tfrac{\pi}{2}]$, and call the corresponding geodesic
cycle 
 $\Gamma^{(3,a)}$.  For small $a$, the points $y_1,y_3,y_5$ are grouped around the north pole at distance $a$, and $y_4,y_6,y_2$ are their antipodals that group around the south pole, see Figures \ref{fig:3a0}, \ref{fig:opt 3}, and \ref{fig:3circ}. 

As above, $a=\arctan(\sqrt{2})$ stands out since in this case the
points  are the vertices of the (Platonic) octahedron. For $a\approx
1/2$, the geodesic cycle resembles the candidate geodesic $3$-design
cycle in Figure \ref{fig:3chain} obtained by numerical optimization. Again, the limiting case  $a =
\frac{\pi}{2}$ yields a great circle, cf.~Figure \ref{fig:3a0}, and
for  $a\rightarrow 0$, $\Gamma^{(3,a)}$ approximates three equally
spaced great circles that run through the north and the south pole,
Figure \ref{fig:3circ}.

We have now  completed item (iv) for the candidate in Figure \ref{fig:3chain}.

\vspace{2mm}

The following theorem finishes the beautification process for the
candidates in Figures \ref{fig:2chain} and \ref{fig:3chain} and proves
the existence of 
geodesic $2$- and $3$-design cycles.  
\begin{tm}\label{thm:2dt simple}
For $t=2,3$, there are parameters  $a_t\in(0,\frac{\pi}{2})$ such that $\Gamma^{(t,a_t)}$ is a geodesic $t$-design cycle.
\end{tm}
\begin{figure}
\subfigure[$\Gamma^{(2,a)}$ for $a=\frac{\pi}{2}-\frac{1}{10}$]{
\includegraphics[width=.3\textwidth]{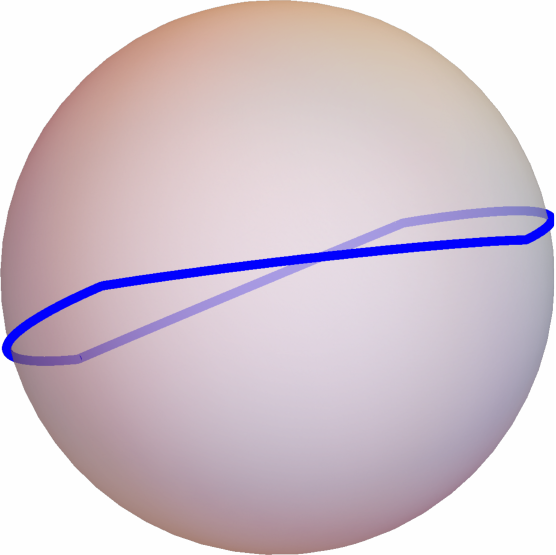}\label{fig:2a0}}\hfill
\subfigure[$\Gamma^{(2,a_2)}$ for $a_2=0.47367\ldots $]{
\includegraphics[width=.3\textwidth]{images/polygonzug_2b.pdf}\label{fig:2circb}}\hfill
\subfigure[$\Gamma^{(2,a)}$ for $a=\frac{1}{30}$]{
\includegraphics[width=.3\textwidth]{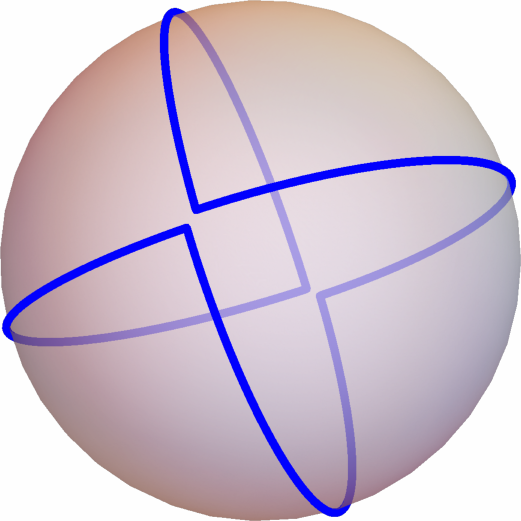}\label{fig:2circ}}

\subfigure[$\Gamma^{(3,a)}$ for $a=\frac{\pi}{2}-\frac{1}{10}$]{
\includegraphics[width=.3\textwidth]{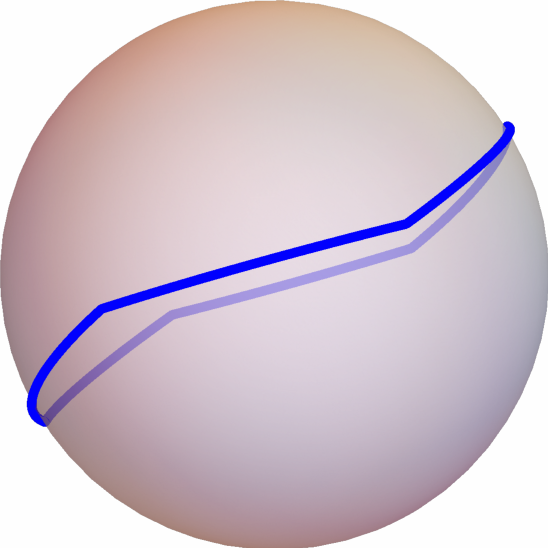}\label{fig:3a0}}\hfill
\subfigure[$\Gamma^{(3,a_3)}$ for $a_3\approx0.449858$]{
\includegraphics[width=.3\textwidth]{images/polygonzug_3b.pdf}\label{fig:opt 3}}\hfill
\subfigure[$\Gamma^{(3,a)}$ for $a=\frac{1}{30}$]{
\includegraphics[width=.3\textwidth]{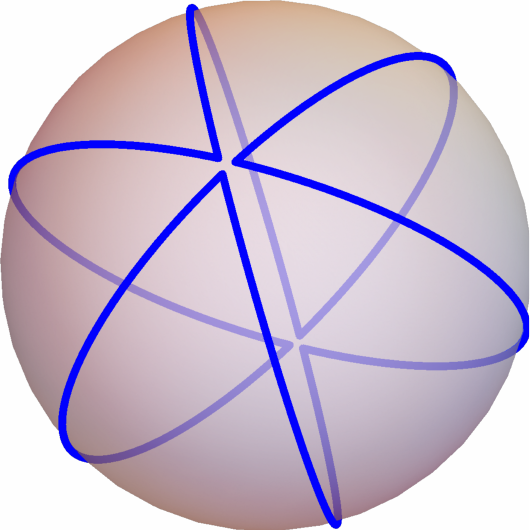}\label{fig:3circ}}
\caption{The geodesic chains $\Gamma^{(2,a)}$ and $\Gamma^{(3,a)}$ in Theorem \ref{thm:2dt simple}. 
}\label{fig:2td}
\end{figure}

\begin{proof}
It suffices to verify the exactness $\frac{1}{\ell(\gamma)}\int_\gamma
f = \int_{\S^2} f$ for the monomials of degree $\leq t$. For the
constant polynomial $f\equiv 1$, this is always satisfied by our normalization. For non-constant monomials, the  symmetries 
of the sphere lead to
\begin{align}
0 & = \int_{\mathbb{S}^2} x^iy^jz^k\,,\quad \text{if  at least one of the $i,j,k$ being odd}\,,\label{eq:degree 1}\\
\frac{1}{3} &= \int_{\S^2}x^2 = \int_{\S^2}y^2 = \int_{\S^2}z^2\,.\label{eq:1/3 xy}
\end{align}

Both families $\Gamma^{(2,a)}$ and $\Gamma^{(3,a)}$ possess some symmetries by \eqref{eq:points t=2 simple} and \eqref{eq:y}, 
 and direct computations reveal that 
 \begin{equation}\label{eq:G2a}
0 = \int_{\Gamma^{(t,a)}} x = \int_{\Gamma^{(t,a)}} y=\int_{\Gamma^{(t,a)}} z= \int_{\Gamma^{(t,a)}} xy = \int_{\Gamma^{(t,a)}} xz =\int_{\Gamma^{(t,a)}} yz\,.
\end{equation}
Likewise, when $t=3$, the curve $\Gamma^{(3,a)}$ is antipodal and therefor the integrals of every monomial of odd degree vanishes. 
Hence, \eqref{eq:degree 1} is matched for these monomials. 
In addition, the identity 
\begin{equation}\label{eq:id x2=y2}
\int_{\Gamma^{(t,a)}} x^2 = \int_{\Gamma^{(t,a)}} y^2
\end{equation}
holds by symmetry of $\Gamma^{(t,a)} $ again. Note that all these
identities hold for arbitrary values of $a\in (0,\pi /2)$. 

We are left with the integral of $x^2, y^2, z^2$ along
$\Gamma^{(t,a)}$. In this case we evaluate the path integral $a\mapsto
\int _{\Gamma^{t,a }}(x^2-z^2)$ for the limiting cases $a=\tfrac{\pi}{2}$ and
$a\to 0$ and will detect a sign change.  For $a=\tfrac{\pi}{2}$, the trace of $\Gamma^{(t,\frac{\pi}{2})}$ is a great circle with
\begin{equation*}
\frac{1}{\ell(\Gamma^{(t,\frac{\pi}{2})})}\int_{\Gamma^{(t,\frac{\pi}{2})}} x^2 = \frac{1}{2}\,,\qquad  \qquad
\frac{1}{\ell(\Gamma^{(t,\frac{\pi}{2})})}\int_{\Gamma^{(t,\frac{\pi}{2})}} z^2 =  0\,,
\end{equation*}
so that we derive $\int _{\Gamma^{t,\frac{\pi}{2} }}(x^2-z^2)>0$.  To verify $\lim_{a\rightarrow 0}\int _{\Gamma^{t,a}}(x^2-z^2)<0$, we restrict ourselves to $t=2$. The case $t=3$ is proved analogously. 

For $a\rightarrow 0$, the trace of $\Gamma^{(2,a)}$ converges towards two great circles through the north and south pole that we now denote by $\Gamma^{(2,0)}$, see Figure \ref{fig:2circ}. One may compute directly or with the help of Mathematica
\begin{equation*}
 \frac{1}{\ell(\Gamma^{(2,0)})}\int_{\Gamma^{(2,0)}} x^2  = \frac{1}{4}\,,\qquad\quad 
\frac{1}{\ell(\Gamma^{(2,0)})}\int_{\Gamma^{(2,0)}} z^2 = \frac{1}{2}\,,
\end{equation*}
so that $\int _{\Gamma^{2,0 }}(x^2-z^2)<0$. 
Since $a \mapsto \frac{1}{\ell(\Gamma^{(2,a)})}\int_{\Gamma^{(2,a)}}
(x^2 -z^2) $ depends continuously on $a$ and changes sign from $a=\frac{\pi}{2}$ to $a=0$, the intermediate value theorem implies that there exists $a_2\in (0,\frac{\pi}{2}]$ with 
\begin{equation}\label{eq:x2z2}
\int_{\Gamma^{(2,a_2)}} (x^2 -z^2) = 0\,.
\end{equation}
Next, the identity $x^2+y^2+z^2=1$ on the sphere $\S ^2$  implies  
    \begin{equation*}
\frac{1}{\ell(\Gamma^{(2,a_2)})}\left(\int_{\Gamma^{(2,a_2)}} x^2 +\int_{\Gamma^{(2,a_2)}} y^2+ \int_{\Gamma^{(2,a_2)}} z^2 \right) = \frac{1}{\ell(\Gamma^{(2,a_2)})}\int_{\Gamma^{(2,a_2)}} 1  = 1 \,.
\end{equation*}
According to  \eqref{eq:id x2=y2} and \eqref{eq:x2z2}, the three integrals on the left-hand-side are equal and, hence, must evaluate to $\frac{1}{3}$. This matches the integration condition \eqref{eq:1/3 xy} and $\Gamma^{(2,a_2)}$ therefore is a geodesic $2$-design chain. 
\end{proof}

\begin{rem}
With the help of  Mathematica, we can  provide the actual nonlinear equation for the parameter $a$ that is mentioned in item (v) of the beautification process.

Using the parameter $\alpha=\sin a$,  the identity  
$\int_{\Gamma^{(2,a)}} (x^2 -z^2) =0$ from
\eqref{eq:x2z2}  can be
expressed  by the equivalent equation 
\begin{equation}\label{eq:a2 parameter}
h_2(\alpha):=\left(2 \alpha^2-1\right) \arccos(\alpha^2-1)-3 \alpha \sqrt{2-\alpha^2} \left(\alpha^2-1\right) = 0\,.
\end{equation}
In the interval $(0,1]$, the function $h_2$ has exactly one root at $\alpha_2=0.456157\ldots$, which leads to $a_2=\arcsin \alpha_2=0.47367\ldots$.

For the case $t=3$, similar computations lead to
\begin{equation}\label{eq:a3 parameter}
h_3(\alpha):=\left(3 \alpha^2-2\right) \arccos(\tfrac{3 \alpha^2}{2}-1)-3 \alpha \sqrt{12-9 \alpha^2} \left(\alpha^2-1\right) = 0\,.
\end{equation}
The only root of $h_3$ within $(0,1]$ is at $\alpha_3=0.434837\ldots$, so that $a_3=\arcsin\alpha_3=0.449858\ldots$. See also Figure \ref{fig:a2+a3}.

Note that the exact algebraic manipulation yields the additional
information that there is a unique parameter $a_t$, such that $\gamma
^{(t,a)}$ is a $t$-design.

\begin{figure}
\includegraphics[width=.45\textwidth]{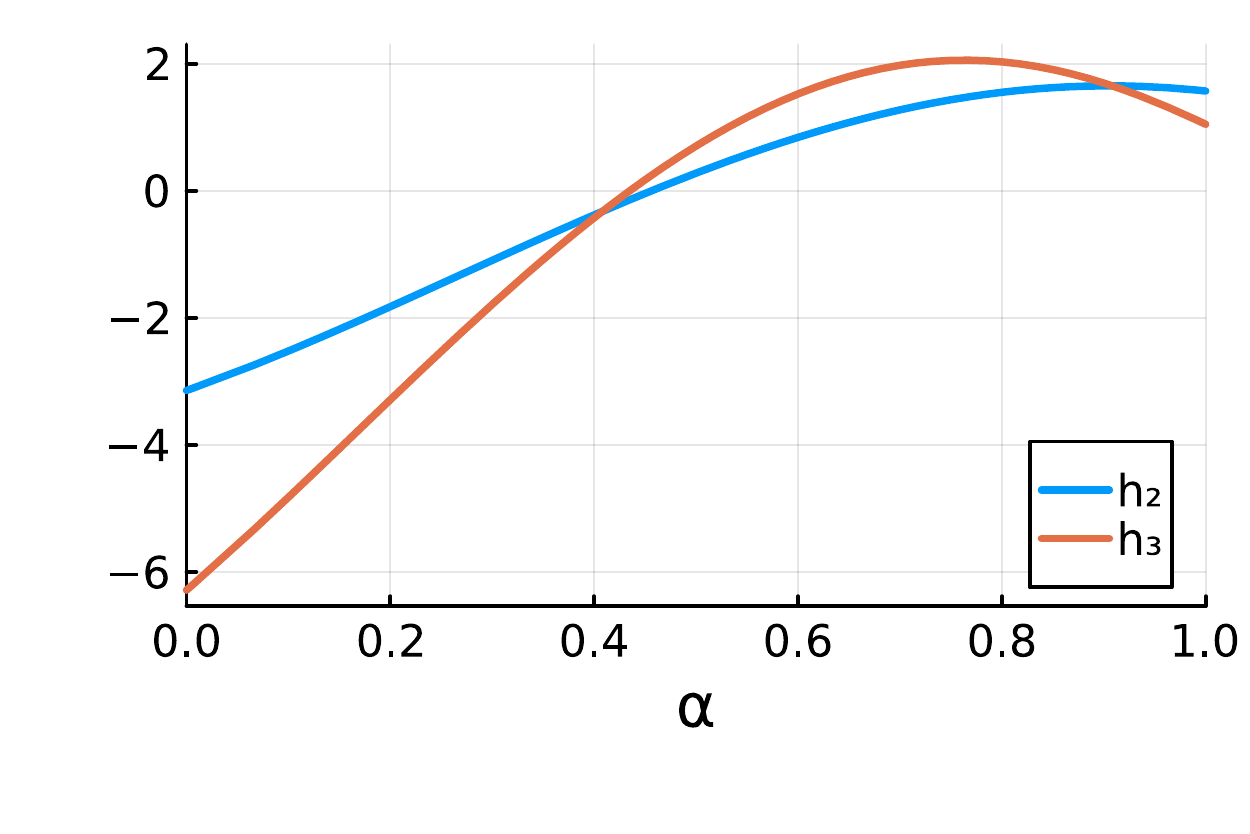}
\caption{Each of the functions $h_2$ and $h_3$ in \eqref{eq:a2 parameter} and \eqref{eq:a3 parameter} has exactly one root in the interval $(0,1]$.}\label{fig:a2+a3}
\end{figure}
\end{rem}

As for Example~\ref{ex:S2}, we observe that  both geodesic cycles $\Gamma^{(2,a)}$ and $
\Gamma^{(3,a)}$   partition the sphere into two regions  of equal
area. This can be seen by using the symmetry properties of these
cycles or, alternatively, 
with the Gauss-Bonnet formula \eqref{eq:GB}. 
Specifically, for geodesic cycles, $\int_\gamma k_g$ is the sum of the
turning angles at the control points. We --- or rather Mathematica ---
have computed the four turning angles of $\Gamma^{(2,a)}$ as 
$$
\pm 1 \pm(\pi+\tfrac{4}{3+\cos(2a)})\,,
$$
where all sign combinations are allowed. They obviously add up  to $0$. 

For $\Gamma^{(3,a)}$, the six turning angles are 
$$
\pm (1 -(\pi+\tfrac{4}{5+3\cos(2a)}))\, ,
$$
where each sign occurs three times, and they also add up to
$0$. Hence, both curves  $\Gamma^{(2,a)}$ and $ \Gamma^{(3,a)}$
partition the sphere into two regions  of equal area.

\subsubsection{Beautifying the candidate resulting from the cube}
To beautify the curve in Figure \ref{fig:3 design candidate cube}
obtained by numerical optimization with the cube as the initial set of
control points, we introduce the following two-parameter family of
eight control points. For $0<\alpha<\beta$ with $\alpha^2+\beta^2<1$
and $q=\sqrt{1-\alpha^2-\beta^2}$, we consider the points 
{\small
\begin{equation}\label{eq:xx}
\begin{aligned}
x_{1}& =
\begin{pmatrix}
\alpha\\ \beta\\q
\end{pmatrix},
&x_{2} &=
\begin{pmatrix}
 \beta\\\alpha\\-q
\end{pmatrix},
&x_{3} &=
\begin{pmatrix}
 -\beta\\\alpha\\-q
\end{pmatrix},
&
x_{4} &=
\begin{pmatrix}
-\alpha\\ \beta\\q
\end{pmatrix},\\
x_{5}& =
\begin{pmatrix}
-\alpha\\ - \beta\\q
\end{pmatrix},
&x_{6} &=
\begin{pmatrix}
- \beta\\-\alpha\\-q
\end{pmatrix},
&x_{7} &=
\begin{pmatrix}
 \beta\\-\alpha\\-q
\end{pmatrix},
&
x_{8} &=
\begin{pmatrix}
\alpha\\ -\beta\\q
\end{pmatrix}\, ,
\end{aligned}
\end{equation}}
and the corresponding  geodesic cycle $\gamma^{(\alpha,\beta)}$. 

For $\alpha=\beta=\frac{1}{\sqrt{3}}$, these points are the vertices of the cube and $\gamma^{(\frac{1}{\sqrt{3}},\frac{1}{\sqrt{3}})}$ is the spherical Hamiltonian cycle of the cube depicted in Figure \ref{fig:cube init}. The choice $\alpha=1/3$ and $\beta=3/4$ leads to a geodesic cycle that resembles the candidate in Figure \ref{fig:3 design candidate cube}. 

The following theorem completes the beautification process of the
curve in Figure \ref{fig:3 design candidate cube} and yields the
existence of 
a $3$-design cycle. 
\begin{tm}
There exist  $\alpha_0\in[\frac{1}{4},\frac{2}{5}]$ and $\beta_0\in[\frac{1}{2},\frac{9}{10}]$, such that $\gamma^{(\alpha_0,\beta_0)}$ is a geodesic $3$-design cycle.
\end{tm}

\begin{proof}
For general $\alpha,\beta$, the geodesic cycle
$\gamma^{(\alpha,\beta)}$ inherits  several symmetries from the
symmetries of  its control
points in \eqref{eq:xx}. Mathematica or direct computations reveal that the integrals
along $\gamma^{(\alpha,\beta)}$ of the odd degree monomials $x$, $y$,
$z$, $xz^2$, $yz^2$, $x^2y$, $xy^2$,  $xyz$, $x^3,y^3,z^3$ 
vanish as required by \eqref{eq:degree 1}. Note that the line
integrals of these monomials along $\gamma^{(\alpha,\beta)}$ vanish
for arbitrary values of $\alpha, \beta $.

For $x^2z$ and $y^2z$, we  still have to satisfy the identities 
\begin{equation}\label{eq:gaga}
\int_{\gamma^{(\alpha,\beta)}}x^2z=\int_{\gamma^{(\alpha,\beta)}}y^2z
=0\,.
\end{equation}
Mathematica computations reveal that the first equality $\int_{\gamma^{(\alpha,\beta)}}x^2z=\int_{\gamma^{(\alpha,\beta)}}y^2z$ holds for all $\alpha,\beta$. To identify distinct parameters such that both integrals vanish, we evaluate the expression in terms of the parameters $\alpha$ and $\beta$ and obtain the first parameter identity
\begin{equation}\label{eq:first one cube}
 (\alpha-\beta) \left(1-\beta^2\right)^{3/2} \sqrt{ 2-(\alpha+\beta)^2} = \beta^3 - \beta^5 + \alpha^2 \beta (\beta^2-3)\,.
\end{equation}

To obtain a second equation for  $\alpha,\beta$, we observe that
$\int_{\gamma^{(\alpha,\beta)}} x^2=\int_{\gamma^{(\alpha,\beta)}}
y^2$ holds for all $\alpha , \beta $. The required  additional identity
$\int_{\gamma^{(\alpha,\beta)}} z^2 = \int_{\gamma^{(\alpha,\beta)}}
x^2$ then leads to the  equation
\begin{equation}\label{eq:parametric two}
\begin{split}
&2 (1 - \beta^2)  (1 - 2 (\alpha^2 - \alpha  \beta + \beta^2))  \arccos((\alpha+\beta)^2-1) \\
&+ (1 - 3 \alpha^2 - \beta^2)  (2 - (\alpha + \beta)^2) \arccos(1 - 2 \beta^2)\\
& \hspace{5cm}=\phantom{rwerewrsdsafsdfsdfsdfsdfs} \\
&  6 \left(1-\alpha^2-\beta^2\right) \!\left( (1-\beta^2)(\alpha+\beta)\sqrt{2-(\alpha+\beta)^2} - \beta \sqrt{1-\beta^2} (2-(\alpha+\beta)^2)\right).
\end{split}
\end{equation}

So far, we have accomplished item (v) of the beautification process by deriving the system of nonlinear equations \eqref{eq:first one cube} and \eqref{eq:parametric two}.
According to item (vi) of the beautification process we now need to
verify that this system of equations is solvable in the parameter range $\alpha^2+\beta^2\leq 1$. 

To use a bivariate version of the intermediate value theorem, we
restrict the parameters to a suitable, smaller  domain  and define the functions $u,v:[\frac{1}{4},\frac{2}{5}] \times[\frac{1}{2},\frac{9}{10}]\rightarrow \mathbb{R}$,
\begin{align}
u(\alpha,\beta) & := (\alpha-\beta) \left(1-\beta^2\right)^{3/2} \sqrt{ 2-(\alpha+\beta)^2} - \beta^3 + \beta^5 - \alpha^2 \beta (\beta^2-3)\,,\label{eq:u}\\
v(\alpha,\beta) & := \text{right-hand-side of \eqref{eq:parametric
                  two} \,--\,
                  left-hand-side of \eqref{eq:parametric
                  two}}\,.\label{eq:v}
\end{align}
The following claims are illustrated in Figure \ref{fig:uv}, but can
also be verified analytically. We observe that $\beta\mapsto
u(\frac{1}{4},\beta)$ is negative and  $\beta\mapsto
u(\frac{2}{5},\beta)$ is positive on $[\frac{1}{2},\frac{9}{10}]$. For
the function $v$, we see that $\alpha\mapsto v(\alpha,\frac{1}{2})$ is
negative and $\alpha\mapsto v(\alpha,\frac{9}{10})$ is positive on the
interval $[\frac{1}{4},\frac{2}{5}]$. As a consequence of Brouwer's
fixed-point theorem, sometimes referred to as the Poincar\'e-Miranda
Theorem \cite{Mawhin19}, there is a tuple $(\alpha_0,\beta_0)\in
[\frac{1}{4},\frac{2}{5}] \times[\frac{1}{2},\frac{9}{10}]$ such that
both functions vanish, $u(\alpha_0,\beta_0)=0$ and
$v(\alpha_0,\beta_0)=0$.

For these parameters both  identity~\eqref{eq:gaga} and the identities
$\int _{\gamma ^{(\alpha _0, \beta _0)}} x^2 =
\int _{\gamma ^{(\alpha _0, \beta _0)}} y^2 = \int _{\gamma ^{(\alpha
    _0, \beta _0)}} z^2 = \tfrac{1}{3}$ are satisfied, consequently 
$\gamma ^{(\alpha _0, \beta _0)}$ is a $3$-design cycle. 
\end{proof}

By  solving  the system of equations \eqref{eq:first one cube} and
\eqref{eq:parametric two} numerically  for $\alpha$ and $\beta$ with
a Newton method with arbitrary precision, we  obtain  
\begin{equation*}
\alpha_0 = 0.381612286088762544249895\ldots ,\qquad \beta_0 = 0.767717328937887399141688\ldots
\end{equation*}

\begin{figure}
\subfigure[$u$]{
\includegraphics[width=.45\textwidth]{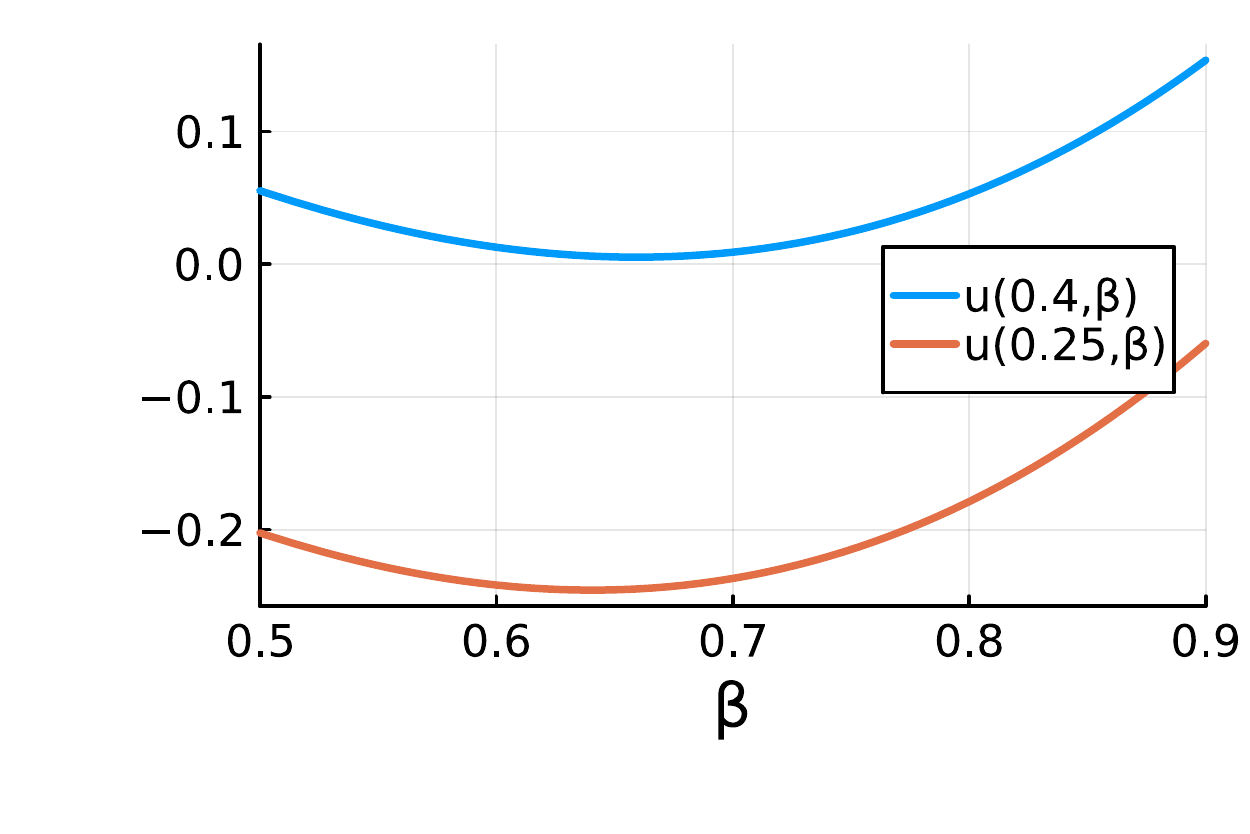}}\hspace{.5cm}
\subfigure[$v$]{
\includegraphics[width=.45\textwidth]{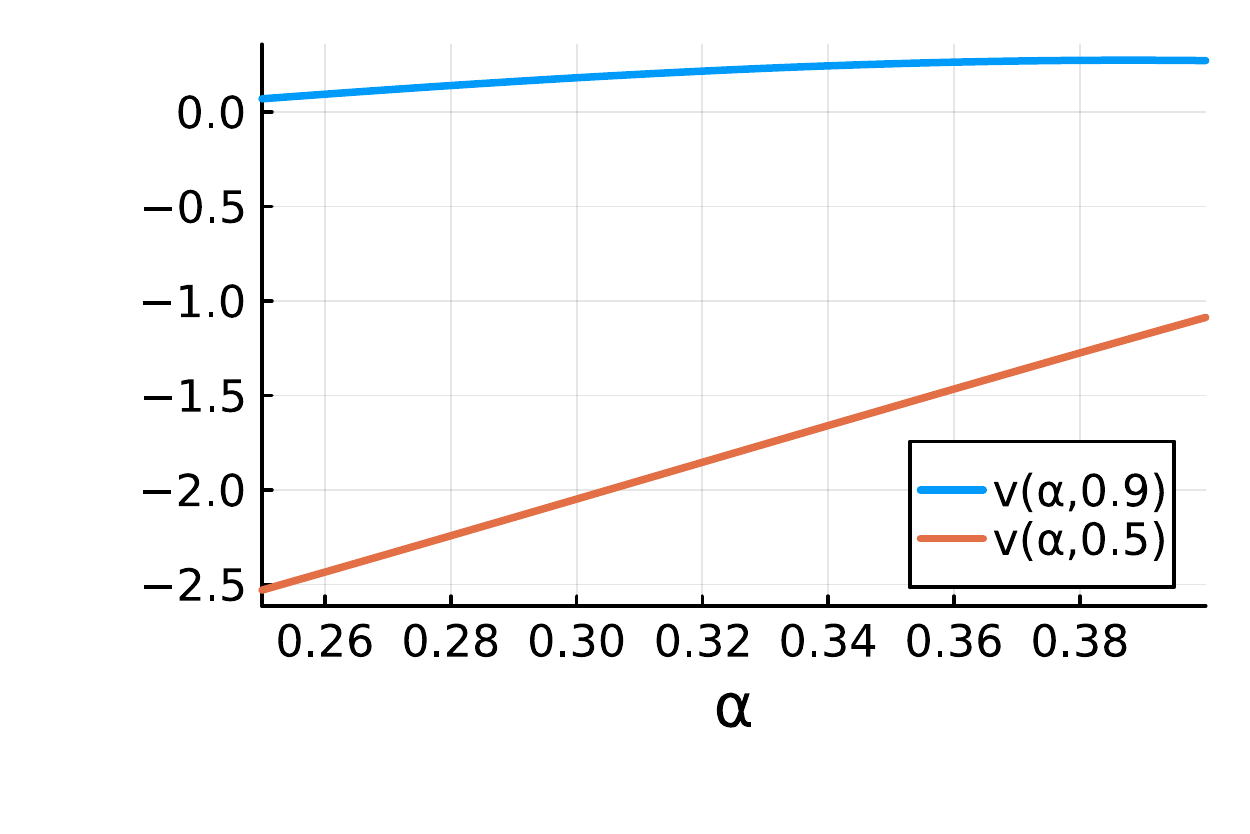}}
\caption{The marginals of the functions $u$ and $v$ in \eqref{eq:u} and \eqref{eq:v}.}\label{fig:uv}
\end{figure}

Again, we may employ symmetry arguments to deduce that the curve $\gamma^{(\alpha,\beta)}$ splits the sphere into two regions of equal area.

\subsubsection{The candidate resulting from the dodecahedron in Figure \ref{fig:cand}.}
The beautification of the numerical candidates of geodesic $5$-design
cycles shown in Figures \ref{fig:3 design candidate dodecahedron} and
\ref{fig:3 design candidate 18 point design} seems more difficult. So
far, we have not been able to complete their beautification process.

\section{\mz\ Inequalities}\label{sec:mz}
We have observed that a closed, simple  curve that connects   $t$-design points along geodesic
arcs is  not  a geodesic $t$-design cycle
in general.  
The idea in  the previous section was to alter and modify this cycle by numerical optimization combined with a beautification procedure that led to $t$-design curves for some $t$. 

In a different direction one may ask what do we obtain if we connect
$t$-design points? In this section we will show that the resulting
geodesic cycles always satisfy \mz\ inequalities. Such inequalities are less rigid, and the construction of curves satisfying \mz\ inequalities is expected to be easier than $t$-design curves. 

\begin{df} \label{defmz}
Fix $1\leq p<\infty$. A sequence of curves $(\gamma_t)_{t\in\N}$ on
the sphere $\S^d$ is called a \mz\ family for $L^p(\S^d)$,  if there
are constants $0<A\leq B<\infty$ such that, for all degrees $t\in\N$
and for   all  polynomials $f\in\Pi_t$  
\begin{equation} \label{mzdef}
A \|f\|_{L^p(\sd)} \leq \Big(\frac{1}{\ell(\gamma_t)} \int_\gamma
|f|^p \Big)^{1/p}\leq B\|f\|_{L^p(\sd)}\, .  
\end{equation}
\end{df}
For $p=\infty $, we use the supremum norm on $\S ^d$ and on $\gamma$, so
that $\|f\|_{L^\infty (\gamma )} = \sup _{s\in [0,1]} |f(\gamma
(s))|$.

The point of Definition~\ref{defmz} is that  the constants $A$
and $B$  are  uniform in $t$ (but they may depend on $p$).

The \mz\ inequalities state the norm equivalence of $\|f\|_{L^p(\S ^d)}$ and
the sampled norm along the curve $\gamma $. For fixed degree $t$ the
inequalities in \eqref{mzdef} hold for all $p$ simultaneously since
$\Pi _t$ is finite-dimensional. This is a new form of discretization of the \( p \)-norm, applied along curves instead of point sets.  
See~\cite{temlyakov} for a first  orientation on sampling and
discretization of norms. 

The connection to $t$-design curves for $p=2$ is explained in the following lemma that expresses the $t$-design property as a norm equality.
\begin{lemma}\label{lemma:design=mz}
A closed curve $\gamma$ is a $2t$-design curve if and only if 
\begin{equation} \label{tight}
  \frac{1}{\ell (\gamma)} \int _{\gamma}  |f|^2  =
  \|f\|_{L^2(\sd)}^2\,  \qquad \text{for all }
  f\in \Pi _t \,.
\end{equation}
\end{lemma}
\begin{proof}
The obvious relation $|f|^2\in\Pi_{2t}$ for $f\in\Pi_t$ implies the
necessity. Conversely, the sufficiency  follows from the polarization
$f g = \frac{1}{4} (f+g)^2 - \frac{1}{4}(f-g)^2$, which implies  the
identity of subspaces 
\begin{equation*}
\spann\{|f|^2 : f\in\Pi_t\}= \spann\{f g : f,g\in\Pi_t\}\,.
\end{equation*}
Clearly the vector space on the right-hand-side coincides with the polynomials $\Pi_{2t}$.
\end{proof}

\subsection{Existence of \mz\ curves}

For fixed degree $t$ it does not seem too difficult to find a curve
that satisfies a \mz\ inequality of the type~\eqref{mzdef}. 
It amounts to choosing $\gamma$ such that $f\in\Pi_t$ and $f\not
\equiv 0$ imply that $f|_\gamma \not \equiv 0$, i.e., $\gamma$ must
not be contained in the zero set of a non-zero polynomial in $\Pi_t$.
For this it suffices that the covering radius of $\gamma $ is small
enough, see \cite[Thm.~2.2]{EG:2023} for this argument. 

However, if we require \mz\ inequalities \eqref{mzdef} for
\emph{all} $t\in\N$ with uniform constants $A_p,B_p>0$ independent of
$t$, then we face a much more challenging problem. 

We now formulate our main result on \mz\ inequalities for curves. In
addition to the existence, the following theorem clarifies how the
constants in the \mz\ inequalities depend on $p$. 
\begin{tm}\label{tm:MZ}
There are constants $0<A_d\leq B_d<\infty$ and for each $\varepsilon\in(0,1)$ there is another constant $C_{d,\varepsilon}>0$ and a sequence of geodesic
cycles $(\gamma^{(t)}_{d,\varepsilon})_{t\in\bN}$ in $\sd$ with the following
properties:

(i) For all $p\in
[1,\infty ]$ and all degrees  $t\in\N$, the norm equivalence 
\begin{align}
A_d^{1/p}(1-\varepsilon)\|f\|_{L^p(\sd)} &\leq  \| f\|_{L^p(\gamma^{(t)}_{d,\varepsilon})}   \leq B_d^{1/p}(1+\varepsilon)  \|f\|_{L^p(\sd)}\,, \label{eq:4} 
\end{align}
holds for all $f\in\Pi_t$, and

(ii) the length of the curves is bounded by
$$
\ell(\gamma^{(t)}_{d,\varepsilon})  \leq C_{d,\varepsilon} t^{d-1}\,.
$$
\end{tm}

For $p=\infty$, we may always take $A_d ^{1/p} =1$ and use the
trivial upper bound $ \| f\|_{L^\infty(\gamma^{(t)}_{d,\varepsilon})}   \leq
\|f\|_{L^\infty(\sd)}$, so that 
the appropriate inequalities in place of  \eqref{eq:4}
are 
\begin{equation*}
(1-\varepsilon)\|f\|_{L^\infty(\sd)} \leq \| f\|_{L^\infty(\gamma^{(t)}_{d,\varepsilon})} \leq \|f\|_{L^\infty(\sd)}\,.
\end{equation*}

The upper bound on the arc length
$\ell(\gamma^{(t)}_{d,\varepsilon})\leq C_{d,\varepsilon} t^{d-1}$ 
means that the \mz\ inequalities are achieved by a
sequence of curves whose arc lengths match the order of the lower bound $\ell(\gamma^{(t)}) \geq c_d  t^{d-1}$ in \eqref{eq:asympt lower bound curve}.

According to Lemma \ref{lemma:design=mz}, every sequence
$(\gamma^{(t)})_{t\in\N}$ of $2t$-design curves satisfies \eqref{eq:4}
for $p=2$. Therefore the existence of \mz\ curves in dimension
$d=2,3$ for  $p=2$ is  covered by the constructions in \cite{EG:2023,Lindblad1}. 

\subsubsection*{Strategy:}
Almost all constructions of \mz\ inequalities start with a partition of the
underlying space, in our case of the $d$-sphere $\sd $. 
Let $\cR = \{ R_1, \dots , R_n\}$ be a partition of $\sd $. We use a
relaxed definition and assume only  that $\bigcup _{j=1}^n R_j = \sd $
and that the intersections $R_j \cap R_k $, for $j\neq k$, have measure zero.

The idea is to start with a sufficiently \emph{nice} partition
$\cR=\{R_1,\ldots,R_n\}$, so that points $\{x_1,\ldots,x_n\}$ with
$x_j\in R_j$ satisfy discrete \mz\ inequalities
$$
(1-\varepsilon) \|f\|_{L^p(\sd)} \leq \left(\frac{1}{n} \sum _{j=1}^n |f(x_j)|^p\right)^{1/p} \leq (1+\varepsilon) \|f\|_{L^p(\sd)}
\, .
$$
We then build a geodesic cycle connecting these points by geodesic
arcs and verify that the \mz\ inequalities for the points induce the
corresponding  inequalities for the curve.

We split the proof of Theorem \ref{tm:MZ} into three  subsections.

\subsection{Partitions}
Following~\cite{Bondarenko:2015eu} a partition $\cR $ is called
area-regular if all patches $R_j$ have the same measure, i.e.,
$|R_j|=1/n$. 
The size of a partition $\cR=\{R_1,\ldots,R_n\}$ is
\begin{equation}
  \label{eq:1}
  \|\cR \| := \max _{j=1, \dots ,n} \mathrm{diam} \,  R_j \, .
\end{equation}

The main theorem about partitions and \mz\ inequalities on the sphere can be
summarized by saying that every choice of points from a sufficiently
fine partition yields a \mz\ inequality. A precise version goes as
follows. See~\cite{FilbirNew,Bondarenko:2011kx,Bondarenko:2015eu,marzo07,marzo08,marzo14,Mhaskar:2002ys}
for several 
variations. 

\begin{tm}\label{thm1}
There exists a
  constant $c_{d}>0$ such that, for every $\varepsilon \in (0,1)$, every area-regular partition $\cR = \{
  R_1, \dots , R_n\} $ of
  size
  \begin{equation}
    \label{eq:2}
    \|\cR \| \leq  c_{d} \varepsilon t^{-1}
  \end{equation}
has the following property: for every  $p\in [1,\infty )$, every collection of points $\{x_1, \dots ,
x_n \}$ with $x_j \in {R}_j$ yields equal-weight \mz\
inequalities~\footnote{\cite{FilbirNew} assumes $x_j$ to be in the interior of $R_j$ but this is not necessary.}
\begin{equation}
  \label{eq:3}
  (1-\varepsilon ) \|f\|_{L^p} \leq \left(\frac{1}{n}\sum _{j=1}^n  |f(x_j)|^p\right)^{1/p} \leq    (1+\varepsilon ) \|f\|_{L^p}\,, \qquad \text{for all }
  f\in \Pi _t \, .
\end{equation}
For $p=\infty$ one has  $ (1-\varepsilon ) \|f\|_{L^\infty} \leq
\max_{j=1,\ldots,n}|f(x_j)| \leq  \|f\|_{L^\infty}$ for all $
  f\in \Pi _t $.
\end{tm}

To match the size requirements \eqref{eq:2} on the partition, we recall a simplified version of the existence results in \cite{Bondarenko:2015eu,Gigante:2017pi}.
\begin{prop} \label{thm2}
There is a constant $C_{\mathrm{diam}}>0$ depending only on the dimension $d$ such that, for every $n\in \bN $, there exists an area-regular partition
  $\cR=\{R_1,\ldots,R_n\} $ of $\sd $ satisfying 
  \begin{equation}\label{eq:Cdiam n}
  \|\cR \| \leq C_{\mathrm{diam}} n^{-1/d}.
  \end{equation}
\end{prop}
Comparing the required   size  $\|\cR \| \leq c_{d} \varepsilon
t^{-1}$  in \eqref{eq:2} with $\|\cR\|\leq C_{\mathrm{diam}} n^{-1/d}$
in \eqref{eq:Cdiam n},  we should choose the
number of patches as 
\begin{equation}\label{eq:t then n}
n \geq  \frac{C_{\diam}^d}{\varepsilon^d c_{d}^d}t^d\,.
\end{equation} 

Note that the constant $\frac{C_{\diam}^d}{\varepsilon^d c_{d}^d}$ only depends on $d$ and $\varepsilon$ but not on $p$.

To transfer the \mz\ inequalities from points to curves, we need to investigate the geometry of the partition more closely. 
We define two types of neighborhoods of a patch $R_j$. Let 
\begin{equation*}
  M_j :=  \{k: \overline{R_k} \cap \overline{R_j} \neq \emptyset \}\,,
\end{equation*}
and define the first neighborhood by all patches that touch or ``kiss'' $R_j$, 
\begin{equation*}
U_j := \bigcup _{k\in M_j} R_k\,.
\end{equation*}
We call $\#M_j$ the kissing number of $R_j$. 

For a convex version of the kissing number,  we denote the closed
geodesic convex hull of $U_k$ by $\mathrm{conv}(U_k)$. By definition,
$\mathrm{conv}(U_k)$ contains all geodesic arcs between points in
$U_k$. Set  
\begin{equation*}
  N_j :=  \{k: \mathrm{conv}(U_k) \cap \overline{R_j} \neq \emptyset \}
\end{equation*}
and define the second type of neighborhood by
\begin{equation*}
V_j:=\bigcup_{k\in N_j} \mathrm{conv}(U_k))\,.
\end{equation*}

The following lemma provides  an upper bound for  the kissing numbers
that depends  only on the dimension $d$ and the constant $C_{\mathrm{diam}}$, but not on the actual partition. 
\begin{lemma}\label{lemma:kiss me}
If $\cR=\{R_1,\ldots,R_n\}$ is an area-regular partition of $\sd$ satisfying  $\|\cR \| \leq C_{\mathrm{diam}} n^{-1/d}$, then 
$$
\#M_j\leq \# N_j \leq C_{\mathrm{kiss}},
$$
for some constant $C_{\mathrm{kiss}}\in\N$ that  depends only on $d$ and $C_{\mathrm{diam}}$. 
\end{lemma}
\begin{proof}
Since $R_k \subseteq U_k$, the condition $\overline{R_k} \cap \overline{R_j} \neq \emptyset$ implies $\mathrm{conv}(R_k) \cap \overline{R_j} \neq \emptyset$ and thus the kissing numbers satisfy
\begin{equation*}
\#M_j\leq \# N_j\,. 
\end{equation*}

Furthermore, since $\mathcal{R}$ is an area-regular partition, we have $n|R_k|=1$ for all $k$ and $R_k\subseteq \mathrm{conv}(U_k)$ yields
\begin{equation}\label{eq:bound for M and V}
\# N_j  =\# N_j n |R_k| = n \Big|\bigcup_{k\in N_j}R_k \Big|\leq n |V_j|\,.
\end{equation}

Next, since by assumption $\|\cR \| \leq C_{\mathrm{diam}} n^{-1/d}$, each $U_k$ is contained in a ball of diameter $3C_{\mathrm{diam}} n^{-1/d}$, and by convexity $\mathrm{conv}(U_k)$ is contained in the same ball. 

We observe that the diameter of the neighborhood $V_j$ is bounded by $2\cdot 3 C_{\mathrm{diam}} n^{-1/d}+C_{\mathrm{diam}} n^{-1/d}=7C_{\mathrm{diam}} n^{-1/d}$. Therefore, its volume satisfies $|V_j| \leq C_{\mathrm{kiss}} n^{-1}$, for some constant $C_{\mathrm{kiss}}>0$ that depends on $d$ and $C_{\mathrm{diam}}$. Therefore, we derive $n |V_j|\leq C_{\mathrm{kiss}}$. Combined with \eqref{eq:bound for M and V}, we obtain $\#N_j\leq C_{\mathrm{kiss}}$.
\end{proof}

Finally we require partitions whose patches contain sufficiently large convex balls. The following existence result \cite{Bondarenko:2015eu,Gigante:2017pi} strengthens Proposition \ref{thm2}. 
\begin{tm}\label{tm:nice partition}
There are constants $c_{\mathrm{in}}, C_{\diam}>0$ depending only on
the dimension $d$, such that, for every $n\in\N$, the following holds: there exists an area-regular partition $\cR=\{R_1,\ldots,R_n\}$ on $\sd$ of size $\|\cR\|\leq C_{\diam} n^{-1/d}$ such that each patch $R_j$ contains a spherical cap of radius $c_{\mathrm{in}}n^{-1/d}$. Furthermore, each patch $R_j$ can be chosen to be convex. 
\end{tm}

\subsection{Construction of the curve from the partition}
We start with a  convex, area-regular partition $\cR $ of size  $\|\cR\|\leq C_{\diam} n^{-1/d}$ 
as in Theorem \ref{tm:nice partition}, so that the discrete  \mz\
inequalities of Theorem~\ref{thm1} hold. Given $t\in\N$, without loss of generality, we may assume 
\begin{equation}\label{eq:nt}
n = \frac{C_{\diam}^d}{\varepsilon^d c_{d}^d}t^d\in\N\,.
\end{equation} 

To this partition $\mathcal{R}$ we associate a graph as follows: its vertices are the patches $R_j$, and a vertex $R_k$ connects to $R_j$ if 
$\overline{R_k} \cap \overline{R_j} \neq \emptyset$. In this case, we
put \emph{two} edges between the vertex $R_j$ and $R_k$. The number of
 edges at $R_j$ is twice the kissing number $2\# M_j$, consequently by
 Lemma~\ref{lemma:kiss me} this graph has bounded degree. 

This graph is connected, see \cite[Lemma 4.1]{EG:2023}. Since we have
doubled the edges, by Euler's criterion on even vertex degree
\cite{Wilson:1998qa}, there is an Euler cycle on the graph. 

By definition, the Euler cycle traverses each edge exactly once and hence visits each patch $R_j$ at least once.   

Since $R_j$ contains a spherical cap of radius $c_{\mathrm{in}}n^{-1/d}$, we may choose \emph{two} points $x_{2j-1},x_{2j}$ in the inner spherical cap of $R_j$ such that
  \begin{equation}
    \label{eq:8}
  \dist(x_{2j-1}, x_{2j}) = c_{\mathrm{in}}n^{-1/d}\,.
  \end{equation}
Since they are contained in a convex subset of $R_j$, the geodesic
arcs $\gamma_j:=\gamma_{x_{2j-1},x_{2j}}$ from $x_{2j-1}$ to $x_{2j}$
are also  contained in $R_j$.  

The Euler cycle induces geodesic arcs
$\tilde{\gamma}_i:[0,1]\rightarrow\sd$ between points in  different patches, but some care is needed. After all, the Euler cycle is a combinatorial object that connects \emph{patches}. To construct the geodesic arcs, we must connect \emph{points} from one patch to another patch, and there is some choice since each patch contains two points.

We start at $R_1$ and agree upon the following construction rules for the arcs.
\begin{itemize}
\item[-] When visiting a patch for the first time, we go along $\gamma_j$, hence, at the beginning along $\gamma_1$. 
\item[-] Following the Euler cycle,  we visit the next patch, say
  $R_j$,   and always arrive at the odd indexed point $x_{2j-1}$.  
\item[-] If we visit $R_j$ for the first time, then we continue along
  $\gamma_j$ to $x_{2j}$. If we have already visited $R_j$ before, then we 
  proceed directly  according to the Euler cycle from $x_{2j-1}$ to $x_{2k-1}$
  in some adjacent patch $R_k$. 
\end{itemize}
For each $t$, the resulting geodesic chain $\gamma$ is a suitable union of arcs $\gamma_j$ contained in $R_j$ and arcs $\tilde{\gamma}_i$ connecting points from adjacent patches.  (To check integration properties, we may ignore any specific ordering, in which the arcs need to be combined). 
 
 The length of each arc $\gamma_j$ in $R_j$ is 
  \begin{equation*}
 \ell(\gamma_j)=c_{\mathrm{in}}n^{-1/d} \,  ,
  \end{equation*}
   and there are $n$ of them. The length of an  arc 
   $\tilde{\gamma}_i$ connecting adjacent patches is  bounded by 
  \begin{equation*}
\ell(\tilde{\gamma}_i)\leq 2C_{\diam} n^{-1/d} \, ,
  \end{equation*}
 and there are at most $2n C_{\mathrm{kiss}}$ of them. Hence, the length of $\gamma$ is sandwiched between 
 \begin{equation*}
 c_{\mathrm{in}}n^{1-1/d}\leq \ell(\gamma)\leq (c_{\mathrm{in}}+2 C_{\mathrm{kiss}}2C_{\diam}) n^{1-1/d}\,.
  \end{equation*}
  Since $n$ is of the order $t^d$ in \eqref{eq:nt}, the length $\ell(\gamma)$ is bounded by a constant times $t^{d-1}$ as claimed in Theorem \ref{tm:MZ}.

  \subsection{The curve $\gamma$ satisfies \mz\ inequalities}
 We fix $f\in\Pi_t$ and aim to verify the inequalities with constants that do not depend on the specific choice of $f$ and are independent of $t$. We only consider $p\in[1,\infty)$ at this point.
  
 The arcs $\gamma _j$ are contained in a convex subset of $R_j$. We use the mean value theorem and obtain parameters $\tau _j$ (that may depend on $f$) 
such that 
\begin{align*}
  \int _{\gamma _j} |f|^p &=  \int _0 ^1 |f(\gamma _j(s))|^p \|\dot{\gamma_j}(s)\|
                             \, ds \\
&= |f(\gamma _j(\tau _j))|^p \ell (\gamma _j) \\
&= |f(y_j)|^p \ell (\gamma _j) \, .
\end{align*}
Here the point $y_j= \gamma _j(\tau _j)$ is in $R_j$ by the convexity of the inner ball in $R_j$. 
We point out that $y_j$ may depend of $f$. Since $y_j\in R_j$ for $j=1, \dots ,n$, Theorem \ref{thm1} implies that the points $\{y_1,
\dots , y_n\}$  satisfy \mz\ inequalities, and therefore
\begin{equation}\label{eq:mzforpoints}
(1-\varepsilon )^p \int _{\sd } |f|^p\leq \frac{1}{n} \sum _{j=1}^n |f(y_j)|^p \leq (1+\varepsilon )^p \int _{\sd } |f|^p
\, .
\end{equation}
Next, we transfer this inequality from points to the curve.
\subsection*{Lower bound for a \mz\ inequality}
To derive the lower bound for a \mz\ inequality for the full curve
$\gamma $, we observe 
\begin{equation*}
\frac{\ell (\gamma _j)}{\ell (\gamma )} \geq
\frac{c_{\mathrm{in}}n^{-1/d}}{(c_{\mathrm{in}}+4 C_{\mathrm{kiss}} C_{\diam}) n^{1-1/d}} = \frac{A_d}{n} \, ,
\end{equation*}
where $A_d=\frac{c_{\mathrm{in}}}{(c_{\mathrm{in}}+4
  C_{\mathrm{kiss}} C_{\diam})}$ depends only on the dimension.

In combination with the lower \mz\ bound \eqref{eq:mzforpoints} we obtain
\begin{align*}
  \frac{1}{\ell (\gamma )} \int _\gamma |f|^p & \geq   \frac{1}{\ell
(\gamma )}\sum _{j=1}^n \int _{\gamma  _j}|f|^p \\
&=  \frac{1}{\ell (\gamma )}\sum _{j=1}^n \ell (\gamma _j) |f(y_j)|^p\\
&\geq A_d \frac{1}{n} \sum _{j=1}^n  |f(y_j)|^p\\
 &\geq A_d (1-\varepsilon )^p \int _{\sd } |f|^p \, .                 
\end{align*}
  
  \subsection*{Upper bound for a \mz\ inequality}
 To verify the upper \mz\
 inequality, we estimate the contribution of the arcs $\gamma _j$  similarly. The fraction of the length 
$$
\frac{\ell (\gamma _j)}{\ell (\gamma )} \leq \frac{c_{\mathrm{in}}n^{-1/d}}{ c_{\mathrm{in}}n^{1-1/d}} = \frac{1}{n}
$$
leads to the upper bound for the contribution of the arcs $\gamma_j$
\begin{align*}
  \frac{1}{\ell (\gamma )} \sum_{j=1}^n\int _{\gamma_j} |f|^p 
  &=  \frac{1}{\ell (\gamma )}\sum _{j=1}^n \ell (\gamma _j) |f(y_j)|^p\\
&\leq \frac{1}{n} \sum _{j=1}^n  |f(y_j)|^p\\
 &\leq  (1+\varepsilon )^p \int _{\sd } |f|^p \, .                 
\end{align*}

 For the contribution of   $\tilde{\gamma } _i$ to the upper
\mz\ bound, we first bound the ratios $\ell(\tilde{\gamma}_i) / \ell(\gamma)$. 
Since $\ell(\tilde{\gamma}_i)\leq 2C_{\diam} n^{-1/d}$ and
$\ell(\gamma)\geq c_{\mathrm{in}}n^{1-1/d}$, we obtain
$$
\frac{\ell(\tilde{\gamma}_i)}{\ell(\gamma)} \leq \frac{2C_{\diam} n^{-1/d}}{c_{\mathrm{in}}n^{1-1/d}}\leq \frac{\tilde{B}_d}{n},
$$
where $\tilde{B}_d=\frac{2C_{\diam}}{c_{\mathrm{in}}}$ depends only on
the dimension $d$. 

The integral along $\tilde{\gamma}_i$ is again evaluated by the mean value theorem, which yields a point $z_i=
\tilde{\gamma}_{i}(\tau '_i)$ such that
$$
\int _{\tilde{\gamma}_i} |f|^p = |f(\tilde{\gamma}_i(\tau '_i))|^p \ell
(\tilde{\gamma}_j) = |f(z_i)|^p \ell (\tilde{\gamma}_i) \, .
$$
For each $\tilde{\gamma}_i$, there is $k\in \{1,\ldots,n\}$ such that $\tilde{\gamma}_i$ is an arc starting from one of the two points in $R_k$, so that  
$z_i$ must lie in $\mathrm{conv}(U_k)$, where $U_k=\bigcup_{\overline{R}_k\cap \overline{R}_l\neq \emptyset}R_l$.  Since by Lemma~\ref{lemma:kiss me} 
$$
\#N_j = \# \{k: \mathrm{conv}(U_k) \cap \overline{R_j} \neq \emptyset \} \leq C_{\mathrm{kiss}}, 
$$
each $R_j$ contains at most $2 C_{\mathrm{kiss}}$ many of the points $z_i$. 

We may add more points, so that every $R_j$ contains $2
C_{\mathrm{kiss}}$ points. This enhanced set is a 
union of $2 C_{\mathrm{kiss}}$ \mz\ sets. Therefore we may use the upper bound of Theorem \ref{thm1} for the points $z_i$ when the bounds are multiplied by $2 C_{\mathrm{kiss}}$. 

The contribution of the $\tilde{\gamma}_i$ is now bounded by
\begin{align*}
\frac{1}{\ell(\gamma)} \sum _{i}  \int _{\tilde{\gamma}_i} |f|^p &\leq  \frac{1}{\ell(\gamma)}\sum _{i}
 |f(z_i)|^p \ell  (\tilde{\gamma}_i) \\
&\leq   \frac{\tilde{B}_d}{n}  \sum _{i}|f(z_i)|^p\\
& \leq \tilde{B}_d 2 C_{\mathrm{kiss}} (1+\varepsilon)^p \int_{\mathbb{S}^d}|f|^p\,.
\end{align*}
Altogether, we obtain the upper \mz\ bound
$$
\frac{1}{\ell(\gamma)}  \int _{\gamma} |f|^p \leq (1+\tilde{B}_d 2 C_{\mathrm{kiss}}) (1+\varepsilon)^p \int_{\mathbb{S}^d}|f|^p\,,\qquad \text{for all }f\in\Pi_t\,,
$$
so that we may choose $B_d=1+\tilde{B}_d 2 C_{\mathrm{kiss}}$.

For the case $p=\infty$, the upper  inequality is trivially satisfied by $\|f\|_{L^\infty(\gamma)}\leq \|f\|_{L^\infty(\sd)}$. To verify the lower bound, we know that the trace of the arc $\gamma_j$ is compact in $\sd$, so that there is $y_j\in R_j$ satisfying 
\begin{equation*}
\max_{s\in[0,1]}|f(\gamma_j(s))| = |f(y_j)|\,,\quad j=1,\ldots,n\,.
\end{equation*}
The lower \mz\ inequality for the points $\{y_j\}_{j=1}^n$ with $p=\infty$ yields
\begin{align*}
\|f\|_{L^\infty(\gamma)} &\geq \max_{j=1,\ldots,n}|f(y_j)| \geq (1-\varepsilon)\|f\|_{L^\infty(\sd)}\,,
\end{align*}
which concludes the proof of Theorem \ref{tm:MZ}.

It remains to verify the simplified version of Theorem \ref{tm:MZ} in the introduction.
\begin{proof}[Proof of Theorem \ref{thm:mz intro}]
We may choose $\varepsilon=\frac{1}{2}$ and $A_d\leq 1\leq B_d$ in Theorem \ref{tm:MZ}, so that the inequalities 
\begin{equation*}
A_d\leq A_d^{1/p}\leq B^{1/p}_d\leq B_d
\end{equation*}
imply Theorem \ref{thm:mz intro} in the introduction with the lower and upper constants $A_d/2$ and $B_d/2$.
\end{proof}

\end{document}

%% file: macros.tex
\newtheorem{tm}{Theorem}[section]    % theorem like environments
\newtheorem{lemma}[tm]{Lemma}
\newtheorem{prop}[tm]{Proposition}

\theoremstyle{definition}
\newtheorem{df}{Definition}[section]
\newtheorem{exmp}{Example}[section]

\newtheorem{rem}{Remark}[section]
\newcommand{\beqa}{\begin{eqnarray*}}
\newcommand{\eeqa}{\end{eqnarray*}}

\newcommand{\field}[1]{\mathbb{#1}}
\newcommand{\bN}{\field{N}}        %  natural numbers 
\def\cR{\mathcal{ R}}

\def\<{\left<}
\def\>{\right>}

\def\mv1{M_v^1}